\documentclass[12pt]{article}
\usepackage{amsthm,amsfonts,amssymb,epsfig,graphics,amsmath,amsbsy,color,enumerate}

\theoremstyle{plain}
\newtheorem{theorem}{Theorem}
\newtheorem{lemma}{Lemma}

\newtheorem{remark}{Remark}
\newtheorem{definition}{Definition}
\newtheorem{corollary}{Corollary}
\newtheorem{claim}{Claim}

\newcommand{\mas}{\operatorname{Mas}}
\newcommand{\mor}{\operatorname{Mor}}

\newcommand{\loc}{\operatorname{loc}}
\newcommand{\diag}{\operatorname{diag}}
\newcommand{\ess}{\operatorname{ess}}
\newcommand{\sgn}{\operatorname{sgn}}
\newcommand{\sech}{\operatorname{sech}}
\newcommand{\discrete}{\operatorname{discrete}}

\numberwithin{equation}{section}
\numberwithin{lemma}{section}
\numberwithin{theorem}{section}
\numberwithin{remark}{section}
\numberwithin{claim}{section}
\numberwithin{corollary}{section}
\numberwithin{proposition}{section}
\numberwithin{definition}{section}
\numberwithin{condition}{section}

\title{The Maslov and Morse Indices for Sturm-Liouville 
Systems on the Half-Line}

\author{Peter Howard and Alim Sukhtayev}

\begin{document}

\maketitle

\begin{abstract} 
We show that for Sturm-Liouville Systems on the half-line 
$[0,\infty)$, the Morse index can be expressed in terms 
of the Maslov index and an additional term 
associated with the boundary conditions at $x = 0$. 
Relations are given both for the case in which the target
Lagrangian subspace is associated with the space of 
$L^2 ((0,\infty), \mathbb{C}^{n})$ 
solutions to the Sturm-Liouville System, and the case when the target 
Lagrangian subspace is associated with the space of 
solutions satisfying the boundary conditions at $x = 0$. 
In the former case, 
a formula of H\"ormander's is used to show that the 
target space can be replaced with the Dirichlet space, 
along with additional explicit terms. We illustrate our 
theory by applying it to an eigenvalue problem that
arises when the nonlinear Schr\"odinger equation on 
a star graph is linearized about a half-soliton 
solution. 
\end{abstract}

\section{Introduction}\label{introduction}

We consider Sturm-Liouville systems
\begin{equation} \label{sturm}
- (P(x) \phi')' + V(x) \phi = \lambda Q(x) \phi, \quad x \in (0,\infty),
\end{equation}
with the one-sided self-adjoint boundary conditions 
\begin{equation} \label{sturm-bc}
\alpha_1 \phi(0) + \alpha_2 P(0) \phi'(0) = 0.
\end{equation}
Here, $\phi = \phi(x; \lambda) \in \mathbb{C}^n$, and we assume:

\medskip
\noindent
{\bf (A1)} The matrices $P (x)$, $V(x)$, and $Q(x)$ are defined
and self-adjoint for a.e. $x \in [0, \infty)$, with also 
$P \in AC_{\loc} ([0,\infty), \mathbb{C}^{n \times n})$ and 
$V(\cdot), Q(\cdot) \in L_{\loc}^1 ([0,\infty), \mathbb{C}^{n \times n})$.
Moreover, there exist constants 
$\theta_P, \theta_Q > 0$ and $C_V \ge 0$ so that 
\begin{equation*}
(P(x)v, v) \ge \theta_P |v|^2;
\quad (Q(x) v,v) \ge \theta_Q |v|^2;
\quad |(V(x) v, v)| \le C_V |v|^2
\end{equation*} 
for a.e. $x \in (0, \infty)$. Here, $(\cdot, \cdot)$ denotes the standard
inner product on $\mathbb{C}^n$, and $|\cdot|$ denotes the standard norm
on the same space. We emphasize that $x = 0$ is included in our local 
designations, so the boundary condition at $x = 0$ is regular. 

\medskip
\noindent
{\bf (A2)} We assume that $P$, $V$, and $Q$ all approach well-defined 
asymptotic endstates at exponential rate. That is, we assume there exist
self-adjoint matrices $P_+, V_+, Q_+ \in \mathbb{C}^{n \times n}$,
with $P_+, Q_+$ positive definite, 
and constants $C$ and $\eta > 0$ so that 
\begin{equation*}
|P(x) - P_+| \le C e^{- \eta |x|},
\quad \text{a.e. } x \in [0, \infty),
\end{equation*}
and similarly for $V(x)$ and $Q(x)$. In addition, we assume 
$|P'(x)| \le C e^{-\eta |x|}$ for a.e. $x \in (0, \infty)$.

\medskip
\noindent
{\bf (A3)} For the boundary conditions, we take 
$\alpha_1, \alpha_2 \in \mathbb{C}^{n \times n}$, and 
for notational convenience, set 
$\alpha = (\alpha_1 \,\, \alpha_2)$. 
We assume $\operatorname{rank} \alpha = n$,
$\alpha J \alpha^* = 0$, 
which is equivalent to self-adjointness in this case. Here, 
$J$ denotes the standard symplectic matrix
\begin{equation*}
J = 
\begin{pmatrix}
0 & - I_n \\
I_n & 0
\end{pmatrix},
\end{equation*}
with $I_n$ denoting the usual $n \times n$ identity matrix. 

We can think of (\ref{sturm}) in terms of the operator
\begin{equation*}
\mathcal{L} \phi = Q(x)^{-1} \{- (P(x) \phi')' + V(x) \phi \},
\end{equation*}
with which we associate the domain
\begin{equation*}
\begin{aligned}
\mathcal{D} (\mathcal{L}) 
&= \{\phi \in L^2 ((0,\infty), \mathbb{C}^n): \phi, \phi' \in AC_{\loc} ([0,\infty), \mathbb{C}^n), \\
& \mathcal{L} \phi \in L^2 ((0,\infty), \mathbb{C}^n), \, \alpha_1 \phi(0) + \alpha_2 P(0) \phi'(0) = 0\},
\end{aligned}
\end{equation*}
and the inner product 
\begin{equation*}
\langle \phi, \psi \rangle_{Q} := \int_0^1 (Q(x) \phi (x), \psi (x))_{\mathbb{C}^n} dx.
\end{equation*}
With this choice of domain and inner product, $\mathcal{L}$ is 
densely defined, closed, and self-adjoint, 
so $\sigma (\mathcal{L}) \subset \mathbb{R}$. 

Our particular interest lies in counting the number of negative
eigenvalues of $\mathcal{L}$ (i.e., the Morse index). We proceed by relating 
the Morse index to the Maslov index, which is described in 
Section \ref{maslov-section}. We find that the 
Morse index can be computed in terms of the Maslov index and
an additional term associated with the boundary condition at $x = 0$.

As a starting point, we define what we will mean by a {\it Lagrangian
subspace} of $\mathbb{C}^{2n}$. For comments about working in 
$\mathbb{C}^{2n}$ rather than $\mathbb{R}^{2n}$, the reader is 
referred to Remark 1.1 of \cite{HS2}, and the references mentioned
in that remark.

\begin{definition} \label{lagrangian_subspace}
We say $\ell \subset \mathbb{C}^{2n}$ is a Lagrangian subspace of $\mathbb{C}^{2n}$
if $\ell$ has dimension $n$ and
\begin{equation} 
(J u, v)_{\mathbb{C}^{2n}} = 0, 
\end{equation} 
for all $u, v \in \ell$. Here, $(\cdot, \cdot)_{\mathbb{C}^{2n}}$ denotes
the standard inner product on $\mathbb{C}^{2n}$. In addition, we denote by 
$\Lambda (n)$ the collection of all Lagrangian subspaces of $\mathbb{C}^{2n}$, 
and we will refer to this as the {\it Lagrangian Grassmannian}. 
\end{definition}

Any Lagrangian subspace of $\mathbb{C}^{2n}$ can be
spanned by a choice of $n$ linearly independent vectors in 
$\mathbb{C}^{2n}$. We will generally find it convenient to collect
these $n$ vectors as the columns of a $2n \times n$ matrix $\mathbf{X}$, 
which we will refer to as a {\it frame} for $\ell$. Moreover, we will 
often coordinatize our frames as $\mathbf{X} = {X \choose Y}$, where $X$ and $Y$ are 
$n \times n$ matrices. Following \cite{F} (p. 274), we specify 
a metric on $\Lambda (n)$ in terms of appropriate orthogonal projections. 
Precisely, let $\mathcal{P}_i$ 
denote the orthogonal projection matrix onto $\ell_i \in \Lambda (n)$
for $i = 1,2$. I.e., if $\mathbf{X}_i$ denotes a frame for $\ell_i$,
then $\mathcal{P}_i = \mathbf{X}_i (\mathbf{X}_i^* \mathbf{X}_i)^{-1} \mathbf{X}_i^*$.
We take our metric $d$ on $\Lambda (n)$ to be defined 
by 
\begin{equation*}
d (\ell_1, \ell_2) := \|\mathcal{P}_1 - \mathcal{P}_2 \|,
\end{equation*} 
where $\| \cdot \|$ can denote any matrix norm. We will say 
that a path of Lagrangian subspaces 
$\ell: \mathcal{I} \to \Lambda (n)$ is continuous provided it is 
continuous under the metric $d$. 

Suppose $\ell_1 (\cdot), \ell_2 (\cdot)$ denote continuous paths of Lagrangian 
subspaces $\ell_i: \mathcal{I} \to \Lambda (n)$, for some parameter interval 
$\mathcal{I}$. The Maslov index associated with these paths, which we will 
denote $\mas (\ell_1, \ell_2; \mathcal{I})$, is a count of the number of times
the paths $\ell_1 (\cdot)$ and $\ell_2 (\cdot)$ intersect, counted
with both multiplicity and direction. (In this setting, if we let 
$t_*$ denote the point of intersection (often referred to as a 
{\it conjugate point}), then multiplicity corresponds with the dimension 
of the intersection $\ell_1 (t_*) \cap \ell_2 (t_*)$; a precise definition of what we 
mean in this context by {\it direction} will be
given in Section \ref{maslov-section}.) 

In order to place our analysis in the usual Hamiltonian framework, 
we express (\ref{sturm}) as a first order system 
for $y = {y_1 \choose y_2}$, with $y_1 = \phi$ and 
$y_2 = P(x) \phi'$. We find 
\begin{equation} \label{ode}
y' = \mathbb{A} (x; \lambda) y,
\end{equation}
where 
\begin{equation*}
\mathbb{A} (x; \lambda) =
\begin{pmatrix}
0 & P(x)^{-1} \\
V(x) - \lambda Q(x) & 0
\end{pmatrix},
\end{equation*}
which can be expressed in the standard linear Hamiltonian form 
\begin{equation*}
Jy' = \mathbb{B} (x; \lambda) y,
\end{equation*}
with 
\begin{equation*}
\mathbb{B} (x; \lambda) =
\begin{pmatrix}
\lambda Q (x) - V(x) & 0 \\
0 & P(x)^{-1}
\end{pmatrix}. 
\end{equation*}

Let $\mathbf{X}_1 (x; \lambda) \in \mathbb{C}^{2n \times n}$ denote the matrix
solution to 
\begin{equation} \label{frame1}
\begin{aligned}
J \mathbf{X}_1' &= \mathbb{B} (x; \lambda) \mathbf{X}_1 \\
\mathbf{X}_1 (0; \lambda) &= J \alpha^*.
\end{aligned}
\end{equation}
We will verify in Section \ref{proof0} that for each 
$(x, \lambda) \in [0,\infty) \times \mathbb{R}$, $\mathbf{X}_1 (x; \lambda)$
is the frame for a Lagrangian subspace of $\mathbb{C}^{2n}$,
$\ell_1 (x; \lambda)$.
Likewise, let $\mathbf{X}_2 (x; \lambda) \in \mathbb{C}^{2n \times n}$ denote the matrix
solution to 
\begin{equation} \label{frame2}
\begin{aligned}
J \mathbf{X}_2' &= \mathbb{B} (x; \lambda) \mathbf{X}_2 \\
\mathbf{X}_2 (\cdot; \lambda) &\in L^2 ((0, \infty), \mathbb{C}^{2n}).
\end{aligned}
\end{equation}
We will verify in Section \ref{ode_section} that for 
\begin{equation} \label{kappa-defined}
\kappa := \inf_{r \in \mathbb{C}^n \backslash \{0\}} \frac{(V_+ r, r)}{(Q_+ r, r)},
\end{equation}
we have $\sigma_{\ess} (\mathcal{L}) \subset [\kappa, \infty)$, where
$\sigma_{\ess} (\cdot)$ denotes essential spectrum, as defined
in Section \ref{ode_section}. Subsequently, we  verify in Section 
\ref{proof0} that for each 
$(x, \lambda) \in [0,\infty) \times (-\infty, \kappa)$, 
$\mathbf{X}_2 (x; \lambda)$
is the frame for a Lagrangian subspace of $\mathbb{C}^{2n}$,
$\ell_2 (x; \lambda)$, and moreover that for 
any $\lambda \in (-\infty, \kappa)$, the asymptotic space
\begin{equation*}
\ell_2^+ (\lambda) := \lim_{x \to \infty} \ell_2 (x; \lambda)
\end{equation*}
is well-defined and Lagrangian (with convergence in the metric
$d$ described above). Finally, we will establish 
that the map $\ell_2^+: (-\infty, \kappa) \to \Lambda (n)$
is continuous.

There are two different ways in which we can formulate 
a relation between the Maslov index and the Morse index, 
depending upon whether we view $x = 0$ as our target 
or $x = + \infty$ as our target. We state these results
respectively as Theorems \ref{target0} and \ref{target+}.
Prior to these statements, we set some terminology with
the following lemma.

\begin{lemma} \label{boundary-inconjugate-lemma} 
Let Assumptions {\bf (A1)}, {\bf (A2)}, and {\bf (A3)} hold, and 
let $\Lambda_{\infty} \in \mathbb{R}$. Then there exists 
$\lambda_{\infty} > \Lambda_{\infty}$ so that 
\begin{equation*}
\ell_1 (0; -\lambda_{\infty}) \cap \ell_2^+ (-\lambda_{\infty}) = \{0\}.
\end{equation*} 
In this case, we refer to $\lambda_{\infty}$ as {\it boundary 
inconjugate}. 
\end{lemma}

We emphasize that for any 
$\lambda_{\infty} \in \mathbb{R}$, $\ell_1 (0; -\lambda_{\infty})$
is the Lagrangian subspace with frame 
$\mathbf{X}_1 (0; -\lambda_{\infty}) = J \alpha^*$, independent of 
$\lambda_{\infty}$. Likewise, in Theorems \ref{target0} and \ref{target+}
below, the Lagrangian space $\ell_1 (0; \lambda_0)$ is independent of 
$\lambda_0$. In all such cases, the appearance of a spectral 
coordinate is only for notational consistency, since $\ell_1 (x; \lambda)$
does in general depend on $\lambda$ for all $x > 0$.

In the following statements, we use the notation $\mor (\mathcal{L}; \lambda_0)$
to indicate the number of eigenvalues that $\mathcal{L}$ has, including multiplicities,
on the interval $(-\infty, \lambda_0)$.   

\begin{theorem} \label{target0}
Let Assumptions {\bf (A1)}, {\bf (A2)}, and {\bf (A3)} hold, and 
fix any $\lambda_0 < \kappa$ (with $\kappa$ defined in (\ref{kappa-defined})). 
Then there exists a value $\Lambda_{\infty}$ sufficiently
large so that for any boundary inconjugate $\lambda_{\infty} > \Lambda_{\infty}$,
we have
\begin{equation*}
\mor (\mathcal{L}; \lambda_0) = 
\mas (\ell_1 (0; \lambda_0), \ell_2 (\cdot; \lambda_0); [0, \infty))
- \mas (\ell_1 (0; \cdot), \ell_2^+ (\cdot); [-\lambda_{\infty}, \lambda_0]).
\end{equation*}
\end{theorem}

\begin{theorem} \label{target+}
Let Assumptions {\bf (A1)}, {\bf (A2)}, and {\bf (A3)} hold, and 
fix any $\lambda_0 < \kappa$ (with $\kappa$ defined in (\ref{kappa-defined})). 
Then there exists a value $\Lambda_{\infty}$ sufficiently
large so that for any boundary inconjugate $\lambda_{\infty} > \Lambda_{\infty}$,
we have
\begin{equation*}
\mor (\mathcal{L}; \lambda_0) = 
- \mas (\ell_1 (\cdot; \lambda_0), \ell_2^+ (\lambda_0); [0, \infty))
- \mas (\ell_1 (0; \cdot), \ell_2^+ (\cdot); [-\lambda_{\infty}, \lambda_0]).
\end{equation*}
\end{theorem} 

\begin{remark} For Theorem \ref{target+}, the target space $\ell_2^+ (\lambda_0)$
can be replaced by the Dirichlet space $\ell_D$ (with frame $\mathbf{X}_D = {0 \choose I_n}$),
at the cost of additional terms that can be expressed explicitly. See 
Corollary \ref{corollary+}. We also note that by combining Theorems \ref{target0}
and \ref{target+} we see that 
\begin{equation*}
\mas (\ell_1 (0; \lambda_0), \ell_2 (\cdot; \lambda_0); [0, \infty))
 = -  \mas (\ell_1 (\cdot; \lambda_0), \ell_2^+ (\lambda_0); [0, \infty)).
\end{equation*}  
\end{remark}

\section{ODE Preliminaries} \label{ode_section}

In this section, we develop preliminary ODE results that will serve as the 
foundation of our analysis. This development is standard, and follows \cite{ZH}, pp. 
779-781 (see, e.g., \cite{Coppel} for similar 
analyses). We begin by clarifying our terminology.

\begin{definition} \label{spectrum}
We define the point spectrum of $\mathcal{L}$, denoted $\sigma_{\operatorname{pt}} (\mathcal{L})$, as the set
\begin{equation*}
\sigma_{\operatorname{pt}} (\mathcal{L}) = \{\lambda \in \mathbb{R}: \mathcal{L} \phi = \lambda \phi 
\, \, \text{for some} \, \, \phi \in \mathcal{D} (\mathcal{L}) \backslash \{0\}\}.
\end{equation*}
Elements of the point spectrum will be referred to as eigenvalues. We define the essential spectrum of 
$\mathcal{L}$, denoted $\sigma_{\operatorname{ess}} (\mathcal{L})$, 
as the values in $\mathbb{C}$ (and so $\mathbb{R}$, by self-adjointness) that 
are not in the resolvent set of $\mathcal{L}$ and are not isolated eigenvalues
of finite multiplicity.
\end{definition}

We note that the total spectrum is 
$\sigma (\mathcal{L}) = \sigma_{\operatorname{pt}} (\mathcal{L}) \cup \sigma_{\operatorname{ess}} (\mathcal{L})$,
and the {\it discrete} spectrum is defined as 
$\sigma_{\operatorname{discrete}} (\mathcal{L}) = \sigma (\mathcal{L}) \backslash \sigma_{\operatorname{ess}} (\mathcal{L})$.
Since our analysis takes place entirely away from essential spectrum, 
the eigenvalues we are counting are elements of the discrete spectrum.  

If we consider (\ref{sturm}) as $x \to \infty$, we obtain the asymptotic system 
\begin{equation} \label{asymptotic}
- P_{+} \phi'' + V_{+} \phi = \lambda Q_{+} \phi.
\end{equation}
For operators such as $\mathcal{L}$ posed on $\mathbb{R}$, it's 
well-known that the essential spectrum is entirely determined by 
the associated asymptotic problems at $\pm \infty$ (see, for 
example, in \cite{Henry, KP}). As we will verify at the end of 
this section, it's straightforward to show that a similar result 
holds true in the current setting.   In particular, if we look for solutions 
of (\ref{asymptotic}) of the form $\phi (x) = e^{i k x} r$, for some scalar constant $k \in \mathbb{R}$
and (non-zero) constant vector $r \in \mathbb{C}^n$ then the essential spectrum 
will be confined to the allowable values of $\lambda$. For (\ref{asymptotic}),
we find 
\begin{equation*}
(k^2 P_{+} + V_{+}) r = \lambda Q_{+} r,
\end{equation*}
and upon taking an inner product with $r$ we see that 
\begin{equation*}
k^2 (P_{+} r, r) + (V_{+} r, r) = \lambda (Q_{+} r, r).
\end{equation*}
Since $P_{+}$ and $Q_{+}$ are positive definite, we see that 
\begin{equation*}
\lambda(k) \ge \frac{(V_{+}r, r)}{(Q_{+}r, r)},
\end{equation*}
for all $k \in \mathbb{R}$, and consequently   
$\sigma_{\operatorname{ess}} (\mathcal{L}) \subset [\kappa, \infty)$,
where 
\begin{equation*}
\kappa = \inf_{r \in \mathbb{C}^n \backslash \{0\}} \frac{(V_+ r, r)}{(Q_+ r, r)} > 0.
\end{equation*}

In order to describe the Lagrangian subspaces $\ell_2 (x; \lambda)$, we need to 
characterize the solutions of (\ref{frame2}) in $L^2 ((0, \infty), \mathbb{C}^{2n})$. 
As a starting point for this 
characterization, we fix any $\lambda < \kappa$ and look for solutions of 
(\ref{asymptotic}) of the form $\phi (x;\lambda) = e^{\mu x} r$, 
where in this case $\mu$ is a scalar function of $\lambda$, and 
$r$ is a vector function of $\lambda$ (in $\mathbb{C}^n$). Computing
directly, we obtain the relation 
\begin{equation*}
(-\mu^2 P_{+} + V_{+} - \lambda Q_{+}) r = 0,
\end{equation*}
which we can rearrange as 
\begin{equation*}
P_{+}^{-1} (V_+ - \lambda Q_+) r = \mu^2 r.
\end{equation*}
Since $P_+$ is positive definite, we can work with the 
inner product 
\begin{equation} \label{inner-product}
(r,s)_+ := (P_+ r, s)_{\mathbb{C}^n},
\end{equation}
and it's clear that for $\lambda \in \mathbb{R}$, the operator
$P_{+}^{-1} (V_+ - \lambda Q_+)$ is self-adjoint with respect to 
this inner product, and moreover positive definite for 
$\lambda < \kappa$. We conclude that for $\lambda < \kappa$, 
the eigenvalues  $\mu^2$ will be positive real values, and that 
the associated eigenvectors can be chosen to be orthonormal 
with respect to the inner product (\ref{inner-product}).   
For each of the $n$ values of $\mu^2$ (counted with multiplicity),
we can associate two values $\pm \sqrt{\mu^2}$. By a choice of 
labeling, we can split these values into $n$ negative values 
$\{ \mu_k (\lambda)\}_{k=1}^n$ and $n$ positive values 
$\{ \mu_k\}_{k=n+1}^{2n}$ with the correspondence (again, by
labeling convention)
\begin{equation*}
\mu_k (\lambda) = - \mu_{2n+1-k} (\lambda); k = 1,2, \dots, n.
\end{equation*}
For $k=1,2,\dots,n$, we denote by $r_k$ the eigenvector of 
$P_{+}^{-1} (V_+ - \lambda Q_+)$ with associated eigenvalue
$\mu_k^2 = \mu_{2n+1-k}^2$. I.e., 
\begin{equation*}
P_+^{-1} (V_+ - \lambda Q_+) r_k = \mu_k^2 r_k; \quad k = 1, 2, \dots, n.
\end{equation*} 

Recalling (\ref{ode}), we note that under our asymptotic 
assumptions on $P(x)$, $Q(x)$, and $V(x)$, the limit 
\begin{equation*}
\mathbb{A}_+ (\lambda) := \lim_{x \to +\infty} \mathbb{A} (x; \lambda)
\end{equation*} 
is well-defined. The values $\{\mu_k\}_{k=1}^{2n}$ described above comprise a labeling
of the eigenvalues of $\mathbb{A}_+ (\lambda)$. Each of 
these eigenvalues is semi-simple, and so we can associate
them with a choice of eigenvectors $\{\mathbf{r}_k\}_{k=1}^{2n}$
so that 
\begin{equation*}
\mathbb{A}_+ (\lambda) \mathbf{r}_k (\lambda) 
= \mu_k (\lambda) \mathbf{r}_k (\lambda), \quad 
k \in \{1, 2, \dots, 2n\}.
\end{equation*}
We see that for $k = 1, 2, \dots, n$, we have relations 
\begin{equation*}
\mathbf{r}_k = {r_k \choose \mu_k P_+ r_k};
\quad
\mathbf{r}_{n+k} = {r_k \choose - \mu_{k} P_+ r_k}. 
\end{equation*}
If we set 
\begin{equation} \label{R-defined}
R (\lambda) := 
\begin{pmatrix}
r_1 (\lambda) & r_2 (\lambda) & \dots r_n (\lambda)
\end{pmatrix},
\end{equation}
and 
\begin{equation} \label{D-defined}
D (\lambda) = \diag 
\begin{pmatrix}
\mu_1 (\lambda) & \mu_2 (\lambda) & \dots & \mu_n (\lambda)
\end{pmatrix},
\end{equation}
then we can express a frame for the eigenspace of $\mathbb{A}_+ (\lambda)$
associated with negative eigenvalues as $\mathbf{X}_2^+ = {R \choose P_+ R D}$,
and likewise we can express a frame for the eigenspace of $\mathbb{A}_+ (\lambda)$
associated with positive eigenvalues as $\tilde{\mathbf{X}}_2^+ = {R \choose - P_+ R D}$.

\begin{lemma} \label{ode-lemma}
Assume {\bf (A1)} and {\bf (A2)}, and let $\{\mu_k (\lambda)\}_{k=1}^{2n}$
and $\{\mathbf{r}_k (\lambda) \}_{k=1}^{2n}$ be as described above. 
Then there exists a $\lambda$-dependent family of bases $\{\mathbf{y}_k (\cdot; \lambda)\}_{k=1}^n$,
$\lambda \in (-\infty, \kappa)$, 
for the spaces of $L^2 ((0, \infty), \mathbb{C}^{2n})$ solutions of (\ref{ode}), 
chosen so that  
\begin{equation*}
\mathbf{y}_k (x; \lambda) = e^{\mu_k (\lambda) x} (\mathbf{r}_k (\lambda) + \mathbf{E}_k (x; \lambda));
\quad k = 1, 2, \dots, n,
\end{equation*}
where 
\begin{equation*}
 \mathbf{E}_k (x; \lambda) = \mathbf{O} (e^{- \tilde{\eta} x})
\end{equation*} 
for some $\tilde{\eta} > 0$, and where the $\mathbf{O} (\cdot)$ term is 
uniform for $\lambda \in (-\infty, \kappa)$. 

Moreover, a basis $\{\mathbf{y}_k (\cdot; \lambda)\}_{k = n+1}^{2n}$ for the space of 
non-$L^2 ((0,\infty), \mathbb{C}^n)$ solutions of (\ref{ode}) can be chosen so that
\begin{equation*}
\mathbf{y}_k (x; \lambda) = e^{\mu_k (\lambda) x} (\mathbf{r}_{2n+1-k} (\lambda) + \mathbf{E}_k (x; \lambda));
\quad k = n+1, n+2, \dots, 2n
\end{equation*} 
with $\{\mathbf{E}_k (x; \lambda))\}_{k=n+1}^{2n}$ satisfying the same 
properties as $\{\mathbf{E}_k (x; \lambda))\}_{k=1}^{n}$. 

Finally, for each $k \in \{1, 2, \dots, 2n\}$ and each $x > 0$, 
$\mathbf{y}_k (x; \cdot) \in C^1 ((-\infty, \kappa), \mathbb{C}^{2n})$.
\end{lemma}

\begin{proof}
For any $\lambda < \kappa$, we follow \cite{ZH} and write (\ref{ode}) 
as 
\begin{equation} \label{ode-asymptotic}
y' = \mathbb{A}_+ y + \mathbb{E} (x; \lambda) y,
\end{equation}
where 
\begin{equation*}
\mathbb{E} (x; \lambda) = \mathbb{A} (x; \lambda) - \mathbb{A}_+ (\lambda)
= \mathbf{O} (e^{- \eta x}).
\end{equation*}
We can now fix a particular index $k \in \{1, 2, \dots, n\}$, 
and look for solutions to (\ref{ode-asymptotic}) of the form 
\begin{equation*}
y(x; \lambda) = e^{\mu_k (\lambda) x} z (x; \lambda),
\end{equation*}
for which 
\begin{equation} \label{z-equation}
z' = (\mathbb{A}_+ (\lambda) - \mu_k (\lambda)) z + \mathbb{E} (x; \lambda) z.
\end{equation}

Based on $\eta$, let $\eta_1, \eta_2 \in \mathbb{R}_+$ satisfy
$0 < \eta_1 < \eta_2 < \eta$. Then there exists a neighborhood of 
$\lambda$ on which we can define a continuous projector 
$\mathcal{P}_k (\lambda)$ onto the direct sum of all eigenspaces of
$\mathbb{A}_+ (\lambda)$ with eigenvalues $\tilde{\mu}$ satisfying
$\tilde{\mu} < \mu_k - \eta_1$, and likewise a projector 
$\mathcal{Q}_k (\lambda) = I - \mathcal{P}_k (\lambda)$ projecting 
onto the direct sum of all eigenspaces of $\mathbb{A}_+ (\lambda)$
with eigenvalues $\tilde{\mu}$ satisfying
\begin{equation*}
\tilde{\mu} \ge \mu_k - \eta_1 > \mu_k - \eta_2.
\end{equation*} 
For some fixed value $M > 0$ taken sufficiently large, we will look 
for solutions to (\ref{z-equation}) of the form 
\begin{equation} \label{z-integral}
\begin{aligned}
z (x; \lambda) &= \mathbf{r}_k (\lambda) 
+ \int_M^x e^{(\mathbb{A}_+ (\lambda) - \mu_k (\lambda)I) (x-y)} \mathcal{P}_k (\lambda) \mathbb{E} (y; \lambda) z (y; \lambda) dy \\
&- \int_x^{+\infty} e^{(\mathbb{A}_+ (\lambda) - \mu_k (\lambda) I) (x - y)} \mathcal{Q}_k (\lambda) \mathbb{E} (y; \lambda) z (y; \lambda) dy.
\end{aligned}
\end{equation}

We proceed by contraction mapping, defining an operator action $\mathcal{T}z$ 
as the right-hand side of (\ref{z-integral}). For this, we use the following
fact, which is immediate from the definitions of $\mathcal{P}_k$ and 
$\mathcal{Q}_k$ : there exist constants $C_1$ and $C_2$ so that 
\begin{equation} \label{projector-estimates}
\begin{aligned}
|e^{(\mathbb{A}_+ (\lambda) - \mu_k (\lambda)I) (x-y)} \mathcal{P}_k (\lambda)| &\le C_1 e^{- \eta_1 (x-y)} \\
|e^{(\mathbb{A}_+ (\lambda) - \mu_k (\lambda)I) (x-y)} \mathcal{Q}_k (\lambda)| &\le C_2 e^{- \eta_2 (y-x)}. 
\end{aligned}
\end{equation}
We check that $\mathcal{T}$ is a contraction on the space $L^{\infty} ((M, \infty), \mathbb{C}^{2n})$. To see this, 
we note that given any $z, w \in L^{\infty} ((M, \infty), \mathbb{C}^{2n})$, there exist constants
$C_3$ and $C_4$ so that  
\begin{equation*}
\begin{aligned}
|\mathcal{T} (z - w)| &\le C_3 \int_M^x e^{- \eta_1 (x-y)} e^{-\eta y} |z (y) - w(y)| dy 
+ C_4 \int_x^{\infty} e^{- \eta_2 (x-y)} e^{-\eta y} |z (y) - w(y)| dy \\
&\le \|z - w\|_{L^{\infty} ((M, \infty), \mathbb{C}^{2n})} \Big{\{} C_3 \int_M^x e^{- \eta_1 (x-y)} e^{-\eta y} dy 
+ C_4 \int_x^{\infty} e^{- \eta_2 (x-y)} e^{-\eta y} dy \Big{\}} \\
&\le \|z - w\|_{L^{\infty} ((M, \infty), \mathbb{C}^{2n})} 
\Big{\{}
C_3 \frac{e^{-\eta_1 (x - M) -\eta M} - e^{-\eta x}}{\eta - \eta_1}
+ C_4 \frac{e^{-\eta x}}{\eta - \eta_2} \Big{\}}.
\end{aligned}
\end{equation*} 
Combining terms, we see that for some constant $C_5$, 
\begin{equation*}
\|\mathcal{T} (z - w)\|_{L^{\infty} ((M, \infty), \mathbb{C}^{2n})}
\le \|z - w\|_{L^{\infty} ((M, \infty), \mathbb{C}^{2n})} C_5 e^{-\eta M},
\end{equation*}
from which it's clear that by taking $M$ sufficiently large, we can 
ensure that we have a contraction. Invariance of $\mathcal{T}$ on 
$L^{\infty} ((M, \infty), \mathbb{C}^{2n})$ follows similarly, and 
we conclude that there exists a unique
$z \in L^{\infty} ((M, \infty), \mathbb{C}^{2n})$ satifying (\ref{z-integral}).
Upon direct differentiation of (\ref{z-integral}), we see
that $z$ solves (\ref{z-equation}). Solutions to (\ref{z-equation})
are absolutely continuous, so in fact $z \in AC_{\loc} ([M, \infty), \mathbb{C}^n)$.
But then we can continue $z$ from $M$ back to 0 by standard 
ODE continuation, so that we have $z \in AC_{\loc} ([0, \infty), \mathbb{C}^n)$.

We can now substitute $z$ back into (\ref{z-integral}) to obtain 
the asymptotic estimates we're after. Proceeding similarly as 
in our verification that $\mathcal{T}$ is a contraction, we find 
that 
\begin{equation*}
z (x) = \mathbf{r}_k (\lambda) + \mathbf{O} (e^{- \eta_1 x}).
\end{equation*}

Finally, differentiability in $\lambda$ is obtained by 
differentiating (\ref{z-integral}) with respect to $\lambda$
and proceeding with a similar argument for the resulting
integral equation. 
\end{proof}

We see from Lemma \ref{ode-lemma} that for each fixed 
$\lambda \in (-\infty, \kappa)$, we can create a frame 
for the Lagrangian subspace of $L^2 ((0, \infty), \mathbb{C}^{2n})$ solutions
of (\ref{ode}), namely 
\begin{equation*}
\mathbf{X}_2 (x; \lambda) = 
\begin{pmatrix}
\mathbf{y}_1 (x; \lambda) & \mathbf{y}_2 (x; \lambda)
& \cdots & \mathbf{y}_n (x; \lambda)
\end{pmatrix}.
\end{equation*}
If we set 
\begin{equation*}
\mathcal{D} (x; \lambda) = \diag 
\begin{pmatrix} e^{\mu_1 (\lambda) x} & e^{\mu_2 (\lambda) x} & \dots & e^{\mu_n (\lambda) x}
\end{pmatrix},
\end{equation*}
then $\mathbf{X}_2 (x; \lambda)$ can be replaced by the 
frame $\mathbf{X}_2 (x; \lambda) \mathcal{D} (x; \lambda)^{-1}$. From this 
latter frame, it's clear that we can take $x \to \infty$ to 
obtain an asymptotic frame $\mathbf{X}_2^+$ comprising the eigenvectors 
$\{\mathbf{r}_k\}_{k=1}^n$ as its columns. 

We can now verify directly that 
\begin{equation*}
\sigma_{\ess} (\mathcal{L}) = [\kappa, \infty).
\end{equation*}
First, for $\lambda < \kappa$, we can directly construct a 
Green's function $G_{\lambda} (x, \xi)$ satisfying 
$\mathcal{L} G_{\lambda} (x, \xi) = \delta_{\xi} (x)$. In 
particular, we obtain 
\begin{equation*}
G_{\lambda} (x, \xi) =
\begin{cases}
X_1 (x; \lambda) \mathcal{M} (\lambda) X_2 (\xi; \lambda)^* Q(\xi) & 0 < x < \xi \\
X_2 (x; \lambda) \mathcal{M} (\lambda)^* X_1 (\xi; \lambda)^* Q(\xi) & 0 < \xi < x,
\end{cases}
\end{equation*} 
where 
\begin{equation*}
\mathcal{M} (\lambda) = - (\mathbf{X}_2 (\xi; \lambda)^* J \mathbf{X}_1 (\xi; \lambda))^{-1}.
\end{equation*}
(The verification that $\mathcal{M} (\lambda)$ is independent of $\xi$ proceeds
almost precisely as the verification that $\mathbf{X}_1 (\xi; \lambda)$ and 
$\mathbf{X}_2 (\xi; \lambda)$ are Lagrangian subspaces.)

According to Lemma 2.2 in \cite{HS2}, for $\lambda < \kappa$, 
$\mathcal{M} (\lambda)$ exists if and only if the Lagrangian 
subspaces $\ell_1 (\xi; \lambda)$ and $\ell_2 (\xi ; \lambda)$ 
do not intersect, and these Lagrangian subspaces intersect 
if and only if $\lambda$ is an eigenvalue of $\mathcal{L}$ 
(i.e., an element of the point spectrum). Moreover, for 
$\lambda < \kappa$, the frames $\mathbf{X}_1 (\xi; \lambda)$
and $\mathbf{X}_2 (\xi; \lambda)$ are analytic in $\lambda$
(see, e.g., Theorem 2.1 in \cite{Weidmann1987}, and this can
also be seen with an approach essentially identical to our 
proof of Lemma \ref{ode-lemma}). It follows that 
$\mathcal{M} (\lambda)$ is meromorphic in $\lambda < \kappa$,
and so there can be no accumulation of eigenvalues on this
interval. This allows us to conclude in fact that
for $\lambda < \kappa$, $\mathcal{M (\lambda)}$ can only fail
to exist if $\lambda \in \sigma_{\discrete} (\mathcal{L})$.

In the case that $\mathcal{M} (\lambda)$ exists, it can be shown 
(e.g., as in the proof of Proposition 7.1 in \cite{ZH})
that there exist constants $C(\lambda) > 0$, $c(\lambda) > 0$
so that
\begin{equation*}
|G_{\lambda} (x, \xi)| \le C(\lambda) e^{- c(\lambda) |x - \xi|}
\end{equation*}   
for all $0 \le x, \xi < \infty$. We can conclude that for 
any $\lambda < \kappa$ that is not an eigenvalue of $\mathcal{L}$,
the resolvent map 
\begin{equation*}
(\mathcal{L} - \lambda I)^{-1} f = \int_0^{\infty} G_{\lambda} (x, \xi) f(\xi) d\xi
\end{equation*}
defines a bounded, linear operator on $L^2 ((0,\infty), \mathbb{C}^n)$.
In particular, $(-\infty, \kappa) \cap \sigma_{\ess} (\mathcal{L}) = \emptyset$.

Although it's not required for the current analysis, we can also 
readily verify that in fact $\sigma_{\ess} (\mathcal{L}) = [\kappa, \infty)$. In 
order to see this, we first not that for any $\lambda \ge \kappa$,
the matrix $P_+^{-1} (V_+ - \lambda Q_+)$ will have one or more 
non-positive eigenvalues. It follows that $\mathbb{A}_+ (\lambda)$
will have two or more eigenvalues with zero real part. The proof
of Lemma \ref{ode-lemma} proceeds essentially unchanged in this
case, and we see that for $\lambda \ge \kappa$ the space of 
$L^2 ((0,\infty), \mathbb{C}^n)$ solutions of 
$\mathcal{L} \phi = \lambda \phi$ has dimension less than $n$. It 
follows immediately from Theorem 11.4.c of \cite{Weidmann1987}
that $\lambda \in \sigma_{\ess} (\mathcal{L})$ in these cases.

\section{The Maslov Index} \label{maslov-section}

Our framework for computing the Maslov index is adapted from 
Section 2 of \cite{HS2}, and we briefly sketch the main ideas
here. Given any pair of Lagrangian subspaces $\ell_1$ and 
$\ell_2$ with respective frames $\mathbf{X}_1 = {X_1 \choose Y_1}$
and $\mathbf{X}_2 = {X_2 \choose Y_2}$, we consider the matrix
\begin{equation} \label{tildeW}
\tilde{W} := - (X_1 + iY_1)(X_1-iY_1)^{-1} (X_2 - iY_2)(X_2+iY_2)^{-1}. 
\end{equation}
In \cite{HS2}, the authors establish: (1) the inverses 
appearing in (\ref{tildeW}) exist; (2) $\tilde{W}$ is independent
of the specific frames $\mathbf{X}_1$ and $\mathbf{X}_2$ (as long
as these are indeed frames for $\ell_1$ and $\ell_2$); (3) $\tilde{W}$ is 
unitary; and (4) the identity 
\begin{equation} \label{key}
\dim (\ell_1 \cap \ell_2) = \dim (\ker (\tilde{W} + I)).
\end{equation}
Given two continuous paths of Lagrangian subspaces 
$\ell_i: [0, 1] \to \Lambda (n)$, $i = 1, 2$, with 
respective frames $\mathbf{X}_i: [0,1] \to \mathbb{C}^{2n \times n}$,
relation (\ref{key}) allows us to compute the Maslov 
index $\mas (\ell_1, \ell_2; [0,1])$ as a spectral flow
through $-1$ for the path of matrices 
\begin{equation} 
\tilde{W} (t) := - (X_1 (t) + iY_1 (t))(X_1 (t)-iY_1 (t))^{-1} 
(X_2 (t) - iY_2 (t))(X_2 (t)+iY_2 (t))^{-1}. 
\end{equation}

In \cite{HS2}, the authors provide a rigorous definition 
of the Maslov index based on the spectral flow developed 
in \cite{P96}. Here, rather, we give only an intuitive 
discussion. As a starting point, 
if $-1 \in \sigma (\tilde{W} (t_*))$ for some $t_* \in [0, 1]$, 
then we refer to $t_*$ as a conjugate point, and its 
multiplicity is taken to be $\dim (\ell_1 (t_*) \cap \ell_2 (t_*)$, 
which by virtue of (\ref{key}) is equivalent to its 
multiplicity as an eigenvalue of $\tilde{W} (t_*)$. 
We compute the Maslov index $\mas (\ell_1, \ell_2; [0, 1])$ 
by allowing $t$ to run from $0$ to $1$ and incrementing 
the index whenever an eigenvalue crosses $-1$ in the 
counterclockwise direction, while decrementing the index
whenever an eigenvalue crosses $-1$ in the clockwise
direction. These increments/decrements are counted with 
multiplicity, so for example, if a pair of eigenvalues 
crosses $-1$ together in the counterclockwise direction, 
then a net amount of $+2$ is added to the index. Regarding
behavior at the endpoints, if an eigenvalue of $\tilde{W}$
rotates away from $-1$ in the clockwise direction as $t$ increases
from $0$, then the Maslov index decrements (according to 
multiplicity), while if an eigenvalue of $\tilde{W}$
rotates away from $-1$ in the counterclockwise direction as $t$ increases
from $0$, then the Maslov index does not change. Likewise, 
if an eigenvalue of $\tilde{W}$ rotates into $-1$ in the 
counterclockwise direction as $t$ increases
to $1$, then the Maslov index increments (according to 
multiplicity), while if an eigenvalue of $\tilde{W}$
rotates into $-1$ in the clockwise direction as $t$ increases
to $1$, then the Maslov index does not change. Finally, 
it's possible that an eigenvalue of $\tilde{W}$ will arrive 
at $-1$ for $t = t_*$ and stay. In these cases, the 
Maslov index only increments/decrements upon arrival or 
departure, and the increments/decrements are determined 
as for the endpoints (departures determined as with $t=0$,
arrivals determined as with $t = 1$).

One of the most important features of the Maslov index is homotopy invariance, 
for which we need to consider continuously varying families of Lagrangian 
paths. To set some notation, we denote by $\mathcal{P} (\mathcal{I})$ the collection 
of all paths $\mathcal{L} (t) = (\ell_1 (t), \ell_2 (t))$, where 
$\ell_1, \ell_2: \mathcal{I} \to \Lambda (n)$ are continuous paths in the 
Lagrangian--Grassmannian. We say that two paths 
$\mathcal{L}, \mathcal{M} \in \mathcal{P} (\mathcal{I})$ are homotopic provided 
there exists a family $\mathcal{H}_s$ so that 
$\mathcal{H}_0 = \mathcal{L}$, $\mathcal{H}_1 = \mathcal{M}$, 
and $\mathcal{H}_s (t)$ is continuous as a map from $(t,s) \in \mathcal{I} \times [0,1]$
into $\Lambda (n) \times \Lambda (n)$. 
 
The Maslov index has the following properties. 

\medskip
\noindent
{\bf (P1)} (Path Additivity) If $\mathcal{L} \in \mathcal{P} (\mathcal{I})$
and $a, b, c \in \mathcal{I}$, with $a < b < c$, then 
\begin{equation*}
\mas (\mathcal{L};[a, c]) = \mas (\mathcal{L};[a, b]) + \mas (\mathcal{L}; [b, c]).
\end{equation*}

\medskip
\noindent
{\bf (P2)} (Homotopy Invariance) If $\mathcal{L}, \mathcal{M} \in \mathcal{P} (\mathcal{I})$ 
are homotopic, with $\mathcal{L} (a) = \mathcal{M} (a)$ and  
$\mathcal{L} (b) = \mathcal{M} (b)$ (i.e., if $\mathcal{L}, \mathcal{M}$
are homotopic with fixed endpoints) then 
\begin{equation*}
\mas (\mathcal{L};[a, b]) = \mas (\mathcal{M};[a, b]).
\end{equation*} 

Straightforward proofs of these properties appear in \cite{HLS1}
for Lagrangian subspaces of $\mathbb{R}^{2n}$, and proofs in the current setting of 
Lagrangian subspaces of $\mathbb{C}^{2n}$ are essentially identical.

\subsection{Exchanging Target Spaces} \label{target-exchange}

Suppose we have a continuous path of Lagrangian subspaces 
$\ell: [0,1] \to \Lambda (n)$, along with two fixed target
Lagrangian subspaces $\ell_1$ and $\ell_2$. Our goal in this
section is to relate the two Maslov indices 
$\mas (\ell, \ell_1; [0,1])$ and $\mas (\ell, \ell_2; [0,1])$.
This question goes back at least to H\"ormander \cite{Hor71},
and has also been discussed in our primary references 
\cite{Duis76} and \cite{ZWZ2018}. 

We suppose $\ell(0)$ intersects neither $\ell_1$ nor 
$\ell_2$ and likewise $\ell(1)$ intersects neither $\ell_1$ nor $\ell_2$,
and we also suppose $\ell_1 \cap \ell_2 = \{0\}$. Then the difference
\begin{equation*}
\mas (\ell, \ell_2; [0,1]) - \mas (\ell, \ell_1; [0,1]),
\end{equation*}
does not depend on the specific path $\ell: [0,1] \to \Lambda (n)$
(see, e.g., \cite{Duis76, Hor71, ZWZ2018}, as discussed below),
and we define the {\it H\"ormander index}
$s (\ell_1, \ell_2; \ell (0), \ell (1))$ by the relation
\begin{equation} \label{hormander-index}
s (\ell_1, \ell_2; \ell (0), \ell (1))
: = \mas (\ell, \ell_2; [0,1]) - \mas (\ell, \ell_1; [0,1]).
\end{equation}
With slight adjustments for notation, this is equation (2.9)
in \cite{Duis76} and Definition 3.9 in \cite{ZWZ2018}. We 
will evaluate the H\"ormander index with an expression 
from \cite{Hor71}, and for this we need to define an 
associated bilinear form. 

\begin{definition} \label{Hormander_Q-form}
Fix any $\ell_1, \ell_2 \in \Lambda (n)$ with $\ell_1 \cap \ell_2 = \{0\}$.
Then any $n$-dimensional linear subspace $\ell_0 \subset \mathbb{C}^{2n}$
(i.e., $\ell_0$ not necessarily Lagrangian) with $\ell_0 \cap \ell_2 = \{0\}$
can be expressed as 
\begin{equation*}
\ell_0 = \{u + Cu: u \in \ell_1\}
\end{equation*}
for some $2n \times 2n$ matrix $C$ that maps $\ell_1$ to $\ell_2$. We 
define a bilinear form 
\begin{equation*}
Q = Q(\ell_1, \ell_2; \ell_0): \ell_1 \times \ell_1 \to \mathbb{C}
\end{equation*}
by the relation 
\begin{equation*}
Q (u, v) := (JCu, v),
\end{equation*}
for all $u, v \in \ell_1$. 
\end{definition}

\begin{remark} Although we will only utilize the bilinear forms 
$Q$ in combination, it's worth noting how we should interpret the 
meaning of an individual form. Given three Lagrangian subspaces 
$\ell_0$, $\ell_1$, and $\ell_2$, $Q (\ell_1, \ell_2; \ell_0)$
provides information about the relative orientation of these 
three spaces. For the case $n=1$, the nature of this information 
is particularly clear. In that setting, we can associate to 
any Lagrangian subspace $\ell_j$ with frame $\mathbf{X}_j = {X_j \choose Y_j}$
a unique point on $S^1$,
\begin{equation*}
\tilde{W}^D_j = (X_j + i Y_j) (X_j - i Y_j)^{-1}.
\end{equation*}
If $\ell_1 \cap \ell_2 = \{0\}$, then $\tilde{W}^D_1$ and 
$\tilde{W}_2^D$ correspond with distinct points on $S^1$. 
Given any third Lagrangian plane $\ell_0$ distinct 
from both $\ell_1$ and $\ell_2$, $\tilde{W}_0^D$ will 
lie either on the arc going from $\tilde{W}^D_1$ 
to $\tilde{W}^D_2$ in the clockwise direction or 
on the arc going from $\tilde{W}^D_1$ to $\tilde{W}^D_2$
in the counterclockwise direction. In the former case, 
we will have $\sgn Q (\ell_1, \ell_2; \ell_0) = +1$, 
while in the latter case we will have 
$\sgn Q (\ell_1, \ell_2; \ell_0) = -1$. Using this
observation, we can readily derive H\"ormander's
formula ((\ref{hormander-formula}), just below) for 
the case $n = 1$, and it can subsequently be 
established that (\ref{hormander-formula}) is valid
for $n>1$ as well.  
\end{remark}

H\"ormander's Q-form is precisely the form defined in \cite{Duis76}, and aside
from a sign convention is also the same form specified in Section 3.1
of \cite{ZWZ2018}. Suppose  $\ell(0)$ intersects neither $\ell_1$ nor 
$\ell_2$ and likewise $\ell(1)$ intersects neither $\ell_1$ nor $\ell_2$.
Then if $\ell_1 \cap \ell_2 = \{0\}$, H\"ormander's formula for 
$s (\ell_1, \ell_2; \ell (0), \ell (1))$ can be expressed as 
\begin{equation} \label{hormander-formula}
s (\ell_1, \ell_2; \ell (0), \ell (1))
= \frac{1}{2} \Big(\sgn Q (\ell_1, \ell_2; \ell (0)) - \sgn Q (\ell_1, \ell_2; \ell(1))
\Big),
\end{equation}
where $\sgn (\cdot)$ denotes the usual signature of a bilinear form (number of positive
eigenvalues minus the number of negative eigenvalues).
    
One immediate consequence of this formula is that if 
$\ell: [0,1] \to \Lambda (n)$ is a closed path (i.e., 
$\ell(0) = \ell (1)$), then 
$s (\ell_1, \ell_2; \ell (0), \ell (1)) = 0$. We see that 
if $\ell_1 \cap \ell_2 = \{0\}$, then for any closed path 
so that $\ell (0)$ intersects neither $\ell_1$ nor $\ell_2$
and likewise $\ell(1)$ intersects neither $\ell_1$ nor $\ell_2$,
the target space can be changed from $\ell_1$ to $\ell_2$
without affecting the Maslov index. 

In practice, we would often prefer the Dirichlet plane $\ell_D$
as our target (e.g., when the target is Dirichlet, all crossings
will necessarily be in the same direction), and so let's check 
the calculation associated with exchanging a general Lagrangian 
target space $\ell_G$ with the Dirichlet plane. For notational 
convenience, we will think of this the other way around, 
taking $\ell_1 = \ell_D$ and $\ell_2 = \ell_G$ in our general
formulation. Following our general development, we assume 
$\ell_D \cap \ell_G = \{0\}$, and also that $\ell (0)$ intersects 
neither $\ell_D$ nor $\ell_G$ and likewise $\ell(1)$ intersects 
neither $\ell_D$ nor $\ell_G$. Since the analysis of $\ell (0)$
and $\ell (1)$ are the same, we will proceed with each replaced 
by the general notation $\ell_0$. 

Our starting point is to characterize the maps $C: \ell_D \to \ell_G$.
If $u \in \ell_D$, then $u = {0 \choose u_2}$ for some 
$u_2 \in \mathbb{C}^n$, and consequently 
$Cu = {C_{12} \choose C_{22}} u_2$. In particular, if 
$C$ maps onto $\ell_G$, then ${C_{12} \choose C_{22}}$ 
will be a frame for $\ell_G$. We denote the set of all 
such maps $C$ by $\mathcal{C}$. Next, we must be able 
to find some $C^{(0)} \in \mathcal{C}$ so that given any 
$w \in \ell_0$ there will exist $u \in \ell_D$ so that 
$w = u + C^{(0)} u$. I.e., we must have $w = {0 \choose u_2} 
+ {C^{(0)}_{12} \choose C^{(0)}_{22}} u_2$. Under our assumption 
that $\ell_D \cap \ell_0 = \{0\}$, we can only 
have ${C^{(0)}_{12} \choose C^{(0)}_{22}} u_2 = 0$ if $u_2$
is 0, and so ${C^{(0)}_{12} \choose C^{(0)}_{22}}$ is indeed
a frame for $\ell_G$. 

For $u \in \ell_D$, we can now compute 
\begin{equation*}
Q (\ell_D, \ell_G; \ell_0) (u,u) = (J C^{(0)} u, u)_{\mathbb{C}^{2n}}
= (C^{(0)}_{12} u_2, u_2)_{\mathbb{C}^n},  
\end{equation*}
from which it's clear that 
\begin{equation*}
\sgn Q (\ell_D, \ell_G; \ell_0) = \sgn C^{(0)}_{12},
\end{equation*}
and moreover if $C^{(0)}_{12}$ is invertible 
\begin{equation*}
\sgn Q (\ell_D, \ell_G; \ell_0) = \sgn (C^{(0)}_{12})^{-1}.
\end{equation*} 
Since ${C^{(0)}_{12} \choose C^{(0)}_{22}}$ is a frame for 
$\ell_G$, we must have that for any other frame 
$\mathbf{X}_G = {X_G \choose Y_G}$, there exists an invertible 
matrix $M \in \mathbb{C}^{n \times n}$ so that 
\begin{equation*}
{C^{(0)}_{12} \choose C^{(0)}_{22}} = {X_G \choose Y_G} M.
\end{equation*}
Likewise, if $\mathbf{X}_0 = {X_0 \choose Y_0}$ is any 
frame for $\ell_0$, then any other frame for $\ell_0$ 
can be expressed as ${X_0 \choose Y_0}\mathcal{M}$ for 
some invertible matrix $\mathcal{M} \in \mathbb{C}^{n \times n}$.
In this way, we can express the relation 
\begin{equation*}
\ell_0 = \{u+Cu: u\in \ell_D\}
\end{equation*} 
in terms of frames 
\begin{equation} \label{matrix_sys}
{X_0 \choose Y_0}\mathcal{M} = {0 \choose I} + {C^{(0)}_{12} \choose C^{(0)}_{22}}
= {0 \choose I} + {X_G \choose Y_G} M.
\end{equation}

First, if $X_G$ is invertible (as will be the case in the current analysis),
we can write $M = X_G^{-1} X_0 \mathcal{M}$, and subsequently 
\begin{equation*}
Y_0 \mathcal{M} = I + Y_G X_G^{-1} X_0 \mathcal{M}. 
\end{equation*}
I.e., we have $(Y_0 - Y_G X_G^{-1} X_0) \mathcal{M} = I$, from 
which we see that $Y_0 - Y_G X_G^{-1} X_0$ is the inverse of 
$\mathcal{M}$, and so $\mathcal{M} = (Y_0 - Y_G X_G^{-1} X_0)^{-1}$.
We conclude, 
\begin{equation} \label{c120-simple}
C_{12}^{(0)} = X_0 \mathcal{M} = X_0 (Y_0 - Y_G X_G^{-1} X_0)^{-1}.
\end{equation}

\begin{remark} \label{invertible}
We note that in the event that $X_0$ is also invertible, we obtain the
expression 
\begin{equation} \label{C120both}
C_{12}^{(0)} = (Y_0 X_0^{-1} - Y_G X_G^{-1})^{-1},
\end{equation}
so that 
\begin{equation*}
\sgn C_{12}^{(0)} = \sgn (C_{12}^{(0)})^{-1} = \sgn (Y_0 X_0^{-1} - Y_G X_G^{-1}).
\end{equation*}
\end{remark}

On the other hand, suppose $X_G$ is not invertible. In this case, we 
observe that the system (\ref{matrix_sys}) can be combined into the 
form 
\begin{equation} \label{matrix_sys2}
\begin{aligned}
(X_0 + i Y_0) \mathcal{M} &= (X_G + i Y_G) M + iI \\
(X_0 - i Y_0) \mathcal{M} &= (X_G - i Y_G) M - iI
\end{aligned}.
\end{equation}
The advantage of this formulation is simply that we can be 
sure that every matrix is invertible. In this case, we obtain 
from the first equation in (\ref{matrix_sys2}),
\begin{equation*}
\mathcal{M} = (X_0 + i Y_0)^{-1} \{(X_G + i Y_G) M + iI\},
\end{equation*} 
and upon substitution into the second equation in (\ref{matrix_sys2}),
\begin{equation*}
(X_0 - i Y_0) (X_0 + i Y_0)^{-1} \{(X_G + i Y_G) M + iI\} 
= (X_G - i Y_G) M - iI.
\end{equation*}
Rearranging terms, we can write 
\begin{equation*}
\begin{aligned}
\Big[(X_0 - i Y_0) (X_0 + i Y_0)^{-1} &- (X_G - i Y_G) (X_G + i Y_G)^{-1}\Big] (X_G + i Y_G) M \\
&= - i [(X_0 - i Y_0) (X_0 + i Y_0)^{-1} + I].
\end{aligned}
\end{equation*}
By assumption, $\ell_0$ does not intersect $\ell_G$, and so the matrix 
multiplying $M$ on the left (i.e., the entire matrix, including 
the square brackets) must be invertible. In this way, we arrive at 
\begin{equation*}
M = - i (X_G + i Y_G)^{-1}
[\tilde{W}_0 - \tilde{W}_G]^{-1} 
[\tilde{W}_0 + I],
\end{equation*}
where 
\begin{equation*}
\begin{aligned}
\tilde{W}_0 &:= (X_0 - i Y_0) (X_0 + i Y_0)^{-1} \\
\tilde{W}_G &:= (X_G - i Y_G) (X_G + i Y_G)^{-1}.
\end{aligned}
\end{equation*}
In this case, 
\begin{equation} \label{c120}
C_{12}^{(0)} = X_0 M = 
-i X_0 (X_G + i Y_G)^{-1}
[\tilde{W}_0 - \tilde{W}_G]^{-1} 
[\tilde{W}_0 + I]. 
\end{equation}

We summarize these observations in the following lemma.

\begin{lemma} \label{hormander-lemma} 
Fix any $\ell_G \in \Lambda (n)$, and let $\ell_0 \in \Lambda (n)$
be such that $\ell_0 \cap \ell_D = \{0\}$ and 
$\ell_0 \cap \ell_G = \{0\}$. Then 
\begin{equation*}
\sgn Q (\ell_D, \ell_G; \ell_0) = \sgn C_{12}^{(0)},
\end{equation*}
where $C_{12}^{(0)}$ is specified in (\ref{c120}).
Moreover, if $X_G$ is invertible, then $C_{12}^{(0)}$
is given in (\ref{c120-simple}), and if in addition
$X_0$ is invertible then $C_{12}^{(0)}$ is given in (\ref{C120both}).
\end{lemma}

\begin{remark}
We will use these considerations in Section \ref{proof+} to 
establish Corollary \ref{corollary+}.  
\end{remark}

\section{Proof of Theorem \ref{target0}} \label{proof0}

In this section, we prove Theorem \ref{target0}. Our starting point
is to verify that $\mathbf{X}_1 (0; \lambda)$ and $\mathbf{X}_2 (x; \lambda)$
are indeed frames for Lagrangian subspaces. According to Proposition 2.1 
of \cite{HS2}, a matrix $\mathbf{X} \in \mathbb{C}^{2n \times n}$ is 
the frame for a Lagrangian subspace if and only if the following 
two conditions both hold: (1) $\operatorname{rank} (\mathbf{X}) = n$; 
and (2) $\mathbf{X}^* J \mathbf{X} = 0$.

For $\mathbf{X}_1 (0; \lambda)$, we 
have $\mathbf{X}_1 (0; \lambda) = J \alpha^*$. According to 
{\bf (A3)}, $\operatorname{rank} \alpha = n$, and it follows 
immediately that $\operatorname{rank} J \alpha^* = n$. Moreover,  
\begin{equation*}
\mathbf{X}_1 (0; \lambda)^* J \mathbf{X}_1 (0; \lambda)
= (\alpha J^*) J (J \alpha^*) = \alpha J \alpha^* = 0.
\end{equation*} 
For $\mathbf{X}_2 (x; \lambda)$, we fix $\lambda \in (-\infty, \kappa)$ 
and temporarily set $\mathcal{A} (x) := \mathbf{X}_2 (x; \lambda)^* J \mathbf{X}_2 (x; \lambda)$.  
(Our notation here doesn't assert that $\mathcal{A}$ is independent of 
$\lambda$, but rather that $\lambda$ is fixed in the ensuing calculations.)
Since $\mathbf{X}_2 (\cdot; \lambda) \in L^2 ((0,\infty), \mathbb{C}^{2n \times n})$,
we see that 
\begin{equation*}
\lim_{x \to + \infty} \mathcal{A} (x) = 0.
\end{equation*}
In addition, we can compute directly to find, 
\begin{equation*}
\begin{aligned}
\mathcal{A}' (x) &= \mathbf{X}_2' (x; \lambda)^* J \mathbf{X}_2 (x; \lambda)
+ \mathbf{X}_2 (x; \lambda)^* J \mathbf{X}_2' (x; \lambda) \\
& = - (J \mathbf{X}_2' (x; \lambda))^* \mathbf{X}_2 (x; \lambda)
+ \mathbf{X}_2 (x; \lambda)^* J \mathbf{X}_2' (x; \lambda) \\
& = - (\mathbb{B} (x; \lambda) \mathbf{X}_2 (x; \lambda))^* \mathbf{X}_2 (x; \lambda)
+ \mathbf{X}_2 (x; \lambda)^* \mathbb{B} (x; \lambda) \mathbf{X}_2 (x; \lambda) \\
&= 0,
\end{aligned}
\end{equation*}
where in obtaining the final equality to $0$ we have used the fact that $\mathbb{B} (x; \lambda)$
is self-adjoint. Combining these observations, we can conclude that 
$\mathcal{A} (x) \equiv 0$ on $[0, \infty)$. Since this argument holds for 
any $\lambda \in (-\infty, \kappa)$, we conclude that $\mathbf{X}_2 (x; \lambda)$
is the frame for a Lagrangian subspace for any 
$(x, \lambda) \in [0, \infty) \times (-\infty, \kappa)$.

Finally, we recall from Section \ref{ode_section} that the Lagrangian 
subspaces $\ell_2 (x; \lambda)$ with frames $\mathbf{X}_2 (x; \lambda)$
can be extended as $x$ tends to infinity to the Lagrangian subspaces $\ell_2^+ (\lambda)$ 
with frames $\mathbf{X}_2^+ (\lambda) = {R (\lambda) \choose P_+ R(\lambda) D(\lambda)}$.
Here, $R(\lambda)$ and $D(\lambda)$ are specified respectively in 
(\ref{R-defined}) and (\ref{D-defined}). In order to verify that $\ell_2^+ (\lambda)$
is indeed Lagrangian, we compute 
\begin{equation*}
\begin{aligned}
\mathbf{X}_2^+ (\lambda)^* J \mathbf{X}_2^+ (\lambda) 
&= 
\begin{pmatrix}
R(\lambda)^* & (P_+ R(\lambda) D(\lambda))^* 
\end{pmatrix}
\begin{pmatrix}
- P_+ R(\lambda) D(\lambda) \\ R(\lambda) 
\end{pmatrix} \\
&= - R(\lambda)^* P_+ R(\lambda) D(\lambda) + D(\lambda) R(\lambda)^* P_+ R(\lambda),
\end{aligned}
\end{equation*}
where we have observed that $P_+$ and $D(\lambda)$ are self-adjoint. 
Recalling the normalization identity $R(\lambda)^* P_+ R(\lambda) = I$,
we see that $\mathbf{X}_2^+ (\lambda)^* J \mathbf{X}_2^+ (\lambda) = 0$
for all $\lambda < \kappa$, from which we can conclude that 
$\ell_2^+ (\lambda)$ is Lagrangian.

We proceed now by considering the {\it Maslov box}, for which we fix 
$\lambda_0 < \kappa$, and work with a value $\lambda_{\infty}$ 
that will be chosen sufficiently large
during the proof, and certainly large enough so that 
$-\lambda_{\infty} < \lambda_0$. The Maslov box in this case will refer to the 
following sequence of four lines, creating a rectangle 
in the $(\lambda, x)$-plane: we fix $x = 0$ and let 
$\lambda$ run from $-\lambda_{\infty}$ 
to $\lambda_0$ (the bottom shelf); we fix 
$\lambda = \lambda_0$ and let $x$ run from $0$ to $+\infty$
(the right shelf); we fix $x = + \infty$ and let 
$\lambda$ run from $\lambda_0$ to $-\lambda_{\infty}$ (the 
top shelf); and we fix $\lambda = -\lambda_{\infty}$
and let $x$ run from $+\infty$ to $0$ (the left shelf).

For Theorem \ref{target0}, we view the bottom shelf at 
$x = 0$ as our target, and the Lagrangian subspace we 
associate with the target is $\ell_1 (0; \lambda)$, with frame 
$\mathbf{X}_1 (0; \lambda) = J \alpha^*$. Clearly, 
$\ell_1 (0; \lambda)$ does not depend on $\lambda$, 
and $\lambda$ only appears as an argument for 
notational consistency. In this case, the evolving 
Lagrangian subspace is $\ell_2 (x; \lambda)$,
which we recall corresponds with the space of solutions that
decay as $x \to + \infty$. As our frame for $\ell_2 (x; \lambda)$,
we use the matrix $\mathbf{X}_2 (x; \lambda)$ constructed 
in (\ref{frame2}). We set 
\begin{equation} \label{tildeW0}
\begin{aligned}
\tilde{W} (x; \lambda) 
&= - (X_1 (0; \lambda) + i Y_1 (0; \lambda)) (X_1 (0; \lambda) - i Y_1 (0; \lambda))^{-1} \\
& \times (X_2 (x; \lambda) - i Y_2 (x; \lambda)) (X_2 (x; \lambda) + i Y_2 (x; \lambda))^{-1}.  
\end{aligned}
\end{equation}
The Maslov index computed with $\tilde{W} (x; \lambda)$ will detect 
intersections between $\ell_1 (0;\lambda)$ and $\ell_2 (x; \lambda)$.
For expositional convenience, we consider the sides of the Maslov
box in the following order: bottom, top, left, right. 

{\it Bottom shelf.}
Beginning with the bottom shelf, we observe that our Lagrangian subspaces
have been constructed in such a way that conjugate points 
correspond with eigenvalues of $\mathcal{L}$, with the multiplicity 
of $\lambda$ as an eigenvalue of $\mathcal{L}$ matching the 
multiplicity of the intersection. This means that if each conjugate 
point along the bottom shelf has the same direction then the Maslov
index along the bottom shelf will be (up to a sign) a count of the 
total number of eigenvalues that $\mathcal{L}$ has between 
$-\lambda_{\infty}$ and $\lambda_0$. We will show below that as 
$\lambda$ ranges from $-\lambda_{\infty}$ toward $\lambda_0$
on the bottom shelf, the conjugate points are all negatively 
directed, and so the corresponding Maslov index is a negative of
this count. In addition, we will 
show during our discussion of the left shelf that we can choose 
$\lambda_{\infty}$ sufficiently large so that $\mathcal{L}$ has 
no eigenvalues on the interval $(-\infty, -\lambda_{\infty}]$. 
We will be able to conclude, then, that the Maslov index along
the bottom shelf is negative a count of the total number of 
eigenvalues, including multiplicity, that $\mathcal{L}$ has
below $\lambda_0$; i.e., 
\begin{equation} \label{bottom-shelf}
\mas (\ell_1 (0; \cdot), \ell_2 (0; \cdot); [-\lambda_{\infty}, \lambda_0])
= - \mor (\mathcal{L}; \lambda_0).
\end{equation}

According to Lemma 3.1 of \cite{HS2} (also Lemma 4.2 of 
\cite{HLS1}), rotation of the eigenvalues of 
$\tilde{W} (x; \lambda)$ as $\lambda$ varies---for any fixed 
$x \in [0, \infty)$---can be determined from the 
matrix $\mathbf{X}_2 (x; \lambda)^* J \partial_{\lambda} \mathbf{X}_2 (x; \lambda)$
in the following sense: If this matrix is positive definite
at some point $(x_0, \lambda_0)$, then as $\lambda$ increases
through $\lambda_0$ (with $x = x_0$ fixed), all $n$ eigenvalues of 
$\tilde{W} (x; \lambda)$ will monotonically rotate in the 
counterclockwise direction. 

For this calculation, we temporarily set 
\begin{equation*}
\mathcal{B} (x; \lambda) 
= \mathbf{X}_2 (x; \lambda)^* J \partial_{\lambda} \mathbf{X}_2 (x; \lambda),
\end{equation*}
for which we can compute (with prime denoting differentiation with 
respect to $x$) 
\begin{equation*}
\begin{aligned}
\mathcal{B}' (x; \lambda)
&= \mathbf{X}_2' (x; \lambda)^* J \partial_{\lambda} \mathbf{X}_2 (x; \lambda)
+ \mathbf{X}_2 (x; \lambda)^* J \partial_{\lambda} \mathbf{X}_2' (x; \lambda) \\
&= - (J \mathbf{X}_2' (x; \lambda))^* \partial_{\lambda} \mathbf{X}_2 (x; \lambda)
+ \mathbf{X}_2 (x; \lambda)^* \partial_{\lambda} (J \mathbf{X}_2' (x; \lambda)) \\
&= - (\mathbb{B} (x; \lambda) \mathbf{X}_2 (x; \lambda))^* \partial_{\lambda} \mathbf{X}_2 (x; \lambda)
+ \mathbf{X}_2 (x; \lambda)^* \partial_{\lambda} (\mathbb{B} (x; \lambda) \mathbf{X}_2 (x; \lambda)) \\
& = - \mathbf{X}_2 (x; \lambda)^* \mathbb{B} (x; \lambda)  \partial_{\lambda} \mathbf{X}_2 (x; \lambda)
+ \mathbf{X}_2 (x; \lambda)^* \mathbb{B}_{\lambda} (x; \lambda) \mathbf{X}_2 (x; \lambda) \\
& \quad + \mathbf{X}_2 (x; \lambda)^* \mathbb{B} (x; \lambda) \partial_{\lambda} \mathbf{X}_2 (x; \lambda) \\
& = \mathbf{X}_2 (x; \lambda)^* \mathbb{B}_{\lambda} (x; \lambda) \mathbf{X}_2 (x; \lambda).  
\end{aligned}
\end{equation*}
Integrating, we see that 
\begin{equation*}
\mathcal{B} (x; \lambda) 
= - \int_x^{+\infty} \mathbf{X}_2 (y; \lambda)^* \mathbb{B}_{\lambda} (y; \lambda) \mathbf{X}_2 (y; \lambda) dy,
\end{equation*}
where convergence of the integral is assured by the exponential decay of the elements
in our frame $\mathbf{X}_2$. In this case, 
\begin{equation*}
\mathbb{B}_{\lambda} (x; \lambda)
= 
\begin{pmatrix}
Q(x) & 0 \\
0 & 0
\end{pmatrix},
\end{equation*}
so that 
\begin{equation*}
\mathcal{B} (x; \lambda) 
= - \int_x^{+\infty} X_2 (y; \lambda)^* Q(y) X_2 (y; \lambda) dy.
\end{equation*}
This matrix is clearly non-positive (since $Q$ is positive definite), 
and moreover it cannot have $0$ as an eigenvalue, because the associated
eigenvector $v \in \mathbb{C}^n$ would necessarily satisfy 
$X_2 (y; \lambda) v = 0$ for all $y \in [x,\infty)$, and this would contradict
linear independence of the columns of $X_2 (y; \lambda)$ (as solutions of 
(\ref{sturm})). 

Since $\mathcal{B} (x; \lambda)$ is negative definite, we can conclude 
that as $\lambda$ increases, the eigenvalues of $\tilde{W} (x; \cdot)$
rotative monotonically clockwise. It follows immediately that for the bottom 
shelf, (\ref{bottom-shelf}) holds. 

{\it Top shelf.} For the top shelf (obtained in the limit as $x \to + \infty$), 
we set
\begin{equation*}
\tilde{W}^+ (\lambda) := \lim_{x \to + \infty} \tilde{W} (x; \lambda),
\end{equation*}
and note that $\tilde{W}^+ (\lambda)$ detects intersections between 
$\ell_1 (0; \lambda)$ and $\ell_2^+ (\lambda)$.  Our frames
for these Lagrangian subspaces are explicit, 
$\mathbf{X}_1 (0; \lambda) = J \alpha^*$ and 
$\mathbf{X}_2^+ (\lambda) = {R (\lambda) \choose P_+ R (\lambda) D (\lambda)}$,
and we can use these frames to explicitly compute 
$\mas (\ell_1 (0; \cdot), \ell_2^+ (\cdot); [-\lambda_{\infty}, \lambda_0])$.

We observe that the monotonicity that we found along horizontal 
shelves does not immediately carry over to the top shelf
(since that calculation is only valid for $x \in [0,\infty)$).
Nonetheless, we can conclude monotonicity along the top shelf
in the following way: by continuity of our frames, we know that as 
$\lambda$ increases along the top shelf the eigenvalues of 
$\tilde{W}^+ (\lambda)$ cannot rotate in the counterclockwise
direction. Moreover, eigenvalues of $\tilde{W}^+ (\lambda)$ 
cannot remain at $-1$ for any interval of $\lambda$ 
values. In order to clarify this last statement, we observe 
that the Lagrangian subspaces $\ell_1 (0; \lambda)$
and $\ell_2^+ (\lambda)$ intersect if and only if $\lambda$
is an eigenvalue for the constant coefficient equation 
\begin{equation} \label{sturm-cc}
\begin{aligned}
- P_+ \phi'' + V_+ \phi &= \lambda Q_+ \phi \\
\alpha_1 \phi (0) + \alpha_2 P (0) \phi' (0) &= 0.  
\end{aligned}
\end{equation} 
(Due to the appearance of $P(0)$ in the boundary condition 
rather than $P_+$, this equation may not be self-adjoint, 
but that doesn't affect this argument.) If $\lambda$ 
is an eigenvalue of (\ref{sturm-cc}) that is not isolated 
from the rest of the spectrum, then it must be in the 
essential spectrum of (\ref{sturm-cc}), but by an argument 
essentially identical to the one given at the end of 
Section \ref{ode_section}, we see that the essential 
spectrum for (\ref{sturm-cc}) is confined to the 
interval $[\kappa, \infty)$, so there can be no interval 
of eigenvalues below $\kappa$.

{\it Left shelf.} For the left shelf, intersections between 
$\ell_1 (0;\lambda)$ and $\ell_2 (x; \lambda)$ at some 
value $x = s$ will correspond with one or more non-trivial 
solutions to the truncated boundary value problem 
\begin{equation} \label{sturm-s}
\begin{aligned}
\mathcal{L}_s \phi := 
Q(x)^{-1} \Big(-(P(x) \phi')' + V(x) \phi\Big) &= \lambda \phi, \quad x \in (s, \infty) \\
\alpha_1 \phi(s) + \alpha_2 P(s) \phi'(s) &= 0,
\end{aligned}
\end{equation} 
where 
\begin{equation*}
\begin{aligned}
\phi (\cdot; \lambda) \in \mathcal{D} (\mathcal{L}_s) 
&:= \{\phi \in L^2 ((s,\infty), \mathbb{C}^n): \phi, \phi' \in AC_{\loc} ([s,\infty), \mathbb{C}^n), \\
& \mathcal{L} \phi \in L^2 ((s,\infty), \mathbb{C}^n), \, \alpha_1 \phi(s) + \alpha_2 P(s) \phi'(s) = 0\}.
\end{aligned}
\end{equation*}
For this calculation, it's useful to use the projector formulation 
of our boundary conditions, developed in \cite{BK, Kuchment2004}
(see also \cite{HS} for an implentation of this formulation
in circumstances quite similar to those of the current
analysis). Briefly, there exist three orthogonal
(and mutually orthogonal) projection matrices 
$P_D$ (the Dirichlet projection), $P_N$ (the Neumann 
projection), and $P_R = I - P_D - P_N$ (the Robin
projection), and an invertible self-adjoint
operator $\Lambda$ acting on the space $P_R \mathbb{C}^n$
such that the boundary condition 
\begin{equation*}
\alpha_1 \phi (s) + \alpha_2 P(s) \phi '(s) = 0
\end{equation*} 
can be expressed as 
\begin{equation} \label{projectors}
\begin{aligned}
P_D \phi (s) &= 0 \\
P_N P(s) \phi'(s) &= 0 \\
P_R P(s) \phi'(s) &= P_R \Lambda P_R \phi(s).
\end{aligned}
\end{equation}
Moreover, $P_D$ can be constructed as the projection 
onto the kernel of $\alpha_2$ and $P_N$ can be 
constructed as the projection onto the kernel of 
$\alpha_1$. 

Suppose $\lambda$ is an eigenvalue for (\ref{sturm-s}), with 
corresponding eigenvector 
$\phi (\cdot; \lambda) \in \mathcal{D}(\mathcal{L}_s)$, 
and consider an 
$L^2 ((s, \infty), \mathbb{C}^n)$ inner product of $\phi (\cdot; \lambda)$ with (\ref{sturm-s}).
Integrating once by parts, we obtain (suppressing dependence on 
$\lambda$ for notational brevity)
\begin{equation} \label{energy-integral}
\int_s^{+\infty} (P \phi', \phi') dx - (P (s) \phi'(s), \phi (s))
+ \int_s^{+\infty} (V \phi, \phi) dx = \lambda \int_s^{+\infty} (Q \phi, \phi) dx.  
\end{equation}
Using uniform positivity of the matrices $P$ and $Q$, we can assert
that for the positive constants $\theta_P$ and $\theta_Q$ described 
in {\bf (A1)}, we have
\begin{equation*}
\begin{aligned}
\int_s^{+\infty} (P \phi', \phi') dx
&\ge \theta_P \|\phi'\|_{L^2 ((s,\infty), \mathbb{C}^n)}^2 \\
\int_s^{+\infty} (Q \phi, \phi) dx
&\ge \theta_Q \|\phi\|_{L^2 ((s,\infty), \mathbb{C}^n)}^2. 
\end{aligned}
\end{equation*}
In addition, with $C_V$ as described in {\bf (A1)}, we have
\begin{equation*}
\Big| \int_s^{+\infty} (V \phi, \phi) dx \Big|
\le C_V \|\phi\|_{L^2 ((s,\infty), \mathbb{C}^n)}^2.
\end{equation*}

For the boundary term, we can use our projection formulation 
to write 
\begin{equation*}
\begin{aligned}
(P (s) \phi'(s), \phi (s)) 
&= (P (s) \phi'(s), P_D \phi (s) + P_N \phi (s) + P_R \phi (s)) \\
&= (P(s) \phi' (s),  P_N \phi (s) + P_R \phi (s)) \\
&= (P_N P(s) \phi' (s), \phi (s)) + (P_R P(s) \phi' (s), \phi (s)) \\
&= (P_R \Lambda P_R \phi(s), \phi (s)).
\end{aligned}
\end{equation*} 
We have, then, 
\begin{equation*}
|(P (s) \phi'(s), \phi (s)) | = |(P_R \Lambda P_R \phi(s), \phi (s))|
\le C_b |\phi (s)|^2,
\end{equation*}
where $C_b$ depends only on the boundary matrices $\alpha_1$ and $\alpha_2$.
For $\phi (\cdot; \lambda) \in \mathcal{D} (\mathcal{L}_s)$, we can write 
\begin{equation*}
|\phi (s)|^2 = - \int_s^{\infty} \frac{d}{dx} |\phi (x)|^2 dx
= - \int_s^{\infty} (\phi'(x), \phi(x)) + (\phi (x), \phi'(x)) dx,
\end{equation*}
so that the Cauchy-Schwarz inequality leads to 
\begin{equation*}
\begin{aligned}
|\phi (s)|^2 &\le
\int_s^{\infty} 2 |\phi'(x)| |\phi (x)| dx 
\le \int_s^{\infty} \epsilon |\phi' (x)|^2 + \frac{1}{\epsilon} |\phi (x)|^2 dx \\
&= \epsilon \|\phi'\|_{L^2 ((s,\infty), \mathbb{C}^n)} 
+ \frac{1}{\epsilon} \|\phi\|_{L^2 ((s,\infty), \mathbb{C}^n)},
\end{aligned}
\end{equation*}
for any $\epsilon > 0$.

Combining these observations, we see that (\ref{energy-integral}) leads, for 
any $\lambda < 0$, to the inequality
\begin{equation*}
\begin{aligned}
\lambda \theta_Q \|\phi\|_{L^2 ((s,\infty), \mathbb{C}^n)}^2 &\ge \lambda \int_s^{+\infty} (Q \phi, \phi) dx
\ge \theta_P \|\phi'\|_{L^2 ((s,\infty), \mathbb{C}^n)}^2 - C_V \|\phi\|_{L^2 ((s,\infty), \mathbb{C}^n)}^2 \\
&- C_b \Big(\epsilon \|\phi'\|_{L^2 ((s,\infty), \mathbb{C}^n)} + \frac{1}{\epsilon} \|\phi\|_{L^2 ((s,\infty), \mathbb{C}^n)}\Big).
\end{aligned}
\end{equation*}
We choose $\epsilon$ so that $\theta_P - C_b \epsilon \ge 0$ to obtain the 
inequality
\begin{equation*}
\lambda \theta_Q \|\phi\|_{L^2 ((s,\infty), \mathbb{C}^n)}^2  
\ge - \Big(C_V + \frac{C_b}{\epsilon}\Big) \|\phi\|_{L^2 ((s,\infty), \mathbb{C}^n)}^2, 
\end{equation*}
from which we conclude the lower bound 
\begin{equation} \label{ev_lower-bound}
\lambda \ge - \Big( \frac{C_V}{\theta_Q} + \frac{C_b}{\epsilon \theta_Q} \Big).
\end{equation}

We see that we can choose $\lambda_{\infty}$ sufficiently large so that 
$\mathcal{L}_s$ has no eigenvalues $\lambda$ on the interval 
$(-\infty, -\lambda_{\infty})$ for any $s \in [0, \infty)$. Consequently, there can be no conjugate 
points $s \in [0, \infty)$ along a left shelf at $\lambda = -\lambda_{\infty}$.

\begin{remark} We contrast this observation with the case of 
Sturm-Liouville systems on $[0,1]$, for which conjugate points 
are possible on the left shelf. In the $[0,1]$-setting,
if the boundary conditions at either $0$ or $1$ are Dirichlet,
then there are no crossings along the left shelf 
(for $\lambda_{\infty}$ sufficiently large). The boundary 
condition $\phi \in L^2 ((0, \infty), \mathbb{C}^n)$ often has the same 
effect on unbounded domains as Dirichlet conditions have 
on bounded domains, and this is an example of that 
observation.  
\end{remark}

We note that this analysis leaves open the possibility that the 
asymptotic point at $+ \infty$ is conjugate. In the event that it is 
conjugate, $\lambda_{\infty}$ can be increased 
slightly to break the conjugacy. This is an immediate 
consequence of monotonicity along the top shelf, and serves
to establish Lemma \ref{boundary-inconjugate-lemma}. 

{\it Right shelf.} For the right shelf, we leave the Maslov 
index as a computation,
\begin{equation*}
\mas (\ell_1 (0; \lambda_0), \ell_2 (\cdot; \lambda_0); [0, \infty)).
\end{equation*}

Combining these observations, and using catenation of paths along with 
homotopy invariance, we find that the sum 
\begin{equation*}
\text{bottom shelf}
+ \text{right shelf}
+ \text{top shelf}
+ \text{left shelf} = 0,
\end{equation*}
respectively becomes
\begin{equation*}
- \mor (\mathcal{L}; \lambda_0) 
+ \mas (\ell_1 (0; \lambda_0), \ell_2 (\cdot; \lambda_0); [0, \infty))
- \mas (\ell_1 (0; \cdot), \ell_2^+ (\cdot); [-\lambda_{\infty}, \lambda_0])
-0 = 0,
\end{equation*}
and Theorem \ref{target0} is a rearrangement of this equality. \hfill $\square$

\section{Proof of Theorem \ref{target+}} \label{proof+}

We established in our proof of Theorem \ref{target0} that 
$\ell_2 (x; \lambda)$ is Lagrangian for all 
$(x, \lambda) \in [0, \infty) \times (-\infty, \kappa)$,
and we can proceed similarly to verify that the same is
true for $\ell_1 (x; \lambda)$. We omit the details.

As with our proof of Theorem \ref{target0}, we work with 
the Maslov box, but in this case, we place
the top shelf at $x = x_{\infty}$, for $x_{\infty}$
chosen sufficiently large during the analysis. We proceed in 
this way, because the Lagrangian subspace 
\begin{equation*}
\ell_1^+ (\lambda) := \lim_{x \to +\infty} \ell_1 (x; \lambda)
\end{equation*}
(which is well-defined for each $\lambda < \kappa$) is not generally
continuous as a function of $\lambda$. In particular, it is 
discontinuous at each eigenvalue of $\mathcal{L}$ (see \cite{HLS2}
for a discussion in the context of Schr\"odinger operators
on $\mathbb{R}$). 

We will use the Maslov index to detect intersections between 
our evolving Lagrangian subspace $\ell_1 (x; \lambda)$ and 
our target Lagrangian subspace $\ell_2 (x_{\infty}; \lambda)$.
Re-defining $\tilde{W}$ for this section, we now set 
\begin{equation} \label{tildeW+}
\begin{aligned}
\tilde{W} (x; \lambda) 
&= - (X_1 (x; \lambda) + i Y_1 (x; \lambda)) (X_1 (x; \lambda) - i Y_1 (x; \lambda))^{-1} \\
& \times (X_2 (x_{\infty}; \lambda) - i Y_2 (x_{\infty}; \lambda)) (X_2 (x_{\infty}; \lambda) + i Y_2 (x_{\infty}; \lambda))^{-1}.  
\end{aligned}
\end{equation}
For expositional convenience, we consider the sides of the 
Maslov box in the following order: left, top, bottom/right 
(together).  

{\it Left shelf.} In this case, conjugate points $x=s$ along the left
shelf correspond with values $s$ for which $\lambda = - \lambda_{\infty}$
is an eigenvalue for the ODE 
\begin{equation} \label{left-ode+}
\begin{aligned}
- (P(x) \phi')' + V(x) \phi &= \lambda Q(x) \phi; \quad \text{in } (0,s) \\
\alpha_1 \phi (0) &+ \alpha_2 P(0) \phi' (0) = 0 \\
Y_2 (x_{\infty}; \lambda)^* \phi (s) &- X_2 (x_{\infty}; \lambda)^* P(s) \phi' (s) = 0, 
\end{aligned}
\end{equation} 
where for notational brevity we are suppressing dependence 
of $\phi$ on $\lambda$.
By taking $x_{\infty}$ sufficiently large, we can make $X_2 (x_{\infty}; \lambda)$
as close as we like to the invertible matrix $R(\lambda)$, so that in this case
$X_2 (x_{\infty}; \lambda)$ is also invertible, and we can write, 
\begin{equation} \label{boundary-relation}
P(s) \phi' (s) = (X_2 (x_{\infty}; \lambda)^*)^{-1} Y_2 (x_{\infty}; \lambda)^* \phi (s). 
\end{equation}
Moreover, we have 
\begin{equation} \label{RD}
(X_2 (x_{\infty}; \lambda)^*)^{-1} Y_2 (x_{\infty}; \lambda)^*
\approx (R (\lambda)^*)^{-1} D (\lambda) R(\lambda)^* P_+ 
= P_+ R (\lambda) D (\lambda) R(\lambda)^* P_+,
\end{equation}
where the error on this approximation is $\mathbf{O} (e^{- \eta x_{\infty}})$ for 
some $\eta > 0$. The matrix 
\begin{equation*}
P_+ R (\lambda) D (\lambda) R(\lambda)^* P_+
\end{equation*} 
is self-adjoint, and since the entries of $D (\lambda)$ are the negative eigenvalues of 
$\mathbb{A}_+ (\lambda)$, it is negative definite. Also, the entries of 
$D(\lambda)$ approach $-\infty$ as $\lambda$ approaches $-\infty$, so the 
eigenvalues of $P_+ R(\lambda) D(\lambda) R(\lambda)^* P_+$ approach 
$-\infty$ as $\lambda$ approaches $-\infty$.

Let $\phi (x; \lambda)$ denote a solution to (\ref{left-ode+}). Upon taking an 
$L^2 ((0,s), \mathbb{C}^n)$ inner product of $\phi$ with (\ref{left-ode+}),
we obtain 
\begin{equation*}
- \int_0^s ((P(x) \phi')', \phi) dx + \int_0^s (V(x) \phi, \phi) dx 
= \lambda \int_0^s (Q(x) \phi, \phi) dx.
\end{equation*} 
For the first integral in this last expression, we compute 
\begin{equation*}
- \int_0^s ((P(x) \phi')', \phi) dx 
= \int_0^s (P(x) \phi', \phi') dx
- (P(s) \phi' (s), \phi (s)) 
+ (P(0) \phi' (0), \phi (0)).
\end{equation*}
Using (\ref{boundary-relation}), we see that 
\begin{equation*}
- (P(s) \phi' (s), \phi (s)) = 
- ((X_2 (x_{\infty}; \lambda)^*)^{-1} Y_2 (x_{\infty}; \lambda)^* \phi (s), \phi (s)). 
\end{equation*}

For the boundary term at $x = 0$, we proceed using the projectors 
$P_D$, $P_N$, and $P_R$ determined by $\alpha_1$ and $\alpha_2$
(as specified in (\ref{projectors})). Proceeding as in the 
proof of Theorem \ref{target0}, we find 
\begin{equation*}
(P(0) \phi' (0), \phi (0)) = (P_R \Lambda P_R \phi (0), \phi (0)). 
\end{equation*}
Combining these observations, we see that the boundary terms can 
be expressed as 
\begin{equation*}
\begin{aligned}
- &(P(s) \phi' (s), \phi (s)) 
+ (P(0) \phi' (0), \phi (0)) \\
&= 
- ((X_2 (x_{\infty}; \lambda)^*)^{-1} Y_2 (x_{\infty}; \lambda)^* \phi (s), \phi (s))
+ (P_R \Lambda P_R \phi (0), \phi (0)).
\end{aligned}
\end{equation*}
For $s$ sufficiently small, $\phi (s) = \phi (0) + \mathbf{O} (s)$, so that 
we approximately have 
\begin{equation} 
\Big(\Big((X_2 (x_{\infty}; \lambda)^*)^{-1} Y_2 (x_{\infty}; \lambda)^* - P_R \Lambda P_R \Big) \phi(0), \phi(0) \Big),
\end{equation}
which is positive for $x_{\infty}$ and $\lambda_{\infty}$ both chosen sufficiently 
large (by the discussion following (\ref{RD})). We conclude that there exists $s_0 > 0$ sufficiently small so that 
\begin{equation*}
- ((X_2 (x_{\infty}; \lambda)^*)^{-1} Y_2 (x_{\infty}; \lambda)^* \phi (s), \phi (s))
+ (P_R \Lambda P_R \phi (0), \phi (0)) \ge 0,
\end{equation*}
for all $0 < s \le s_0$.

Similarly as in the proof of Theorem \ref{target0}, we have
\begin{equation*}
\begin{aligned}
\int_0^{s} (P \phi', \phi') dx
&\ge \theta_P \|\phi'\|_{L^2 ((0, s), \mathbb{C}^n)}^2; \\
\int_0^{s} (Q \phi, \phi) dx
&\ge \theta_Q \|\phi\|_{L^2 ((0,s), \mathbb{C}^n)}^2; \\
\Big|\int_0^{s} (V \phi, \phi) dx \Big|
&\le C_V \|\phi\|_{L^2 ((0,s), \mathbb{C}^n)}^2.
\end{aligned}
\end{equation*}
For $\lambda < 0$, this allows us to write (still for $0 < s \le s_0$)
\begin{equation*}
\begin{aligned}
\lambda \theta_Q \|\phi\|_{L^2 ((0,s), \mathbb{C}^n)}^2 
& \ge \lambda \int_0^{s} (Q \phi, \phi) dx \\
&\ge \theta_P \|\phi'\|_{L^2 ((0, s), \mathbb{C}^n)}^2
-  C_V \|\phi\|_{L^2 ((0,s), \mathbb{C}^n)}^2,
\end{aligned}
\end{equation*}
from which we can immediately conclude 
\begin{equation*}
\lambda \ge -  \frac{C_V}{\theta_Q},
\end{equation*}
for all $0 < s \le s_0$. 

For $s > s_0$, we scale the independent variable by 
setting 
\begin{equation*}
\xi = \frac{x}{s}; \quad \varphi (\xi) = \phi (x).
\end{equation*}
Our system becomes  
\begin{equation} \label{left-ode+again}
\begin{aligned}
- (P(\xi s) \varphi')' + s^2 V(\xi s) \varphi &= s^2 \lambda Q(\xi s) \varphi; \quad \text{in } (0,1) \\
\alpha_1 \varphi (0) &+ \frac{1}{s} \alpha_2 P(0) \varphi' (0) = 0 \\
Y_2 (x_{\infty}; \lambda)^* \varphi (1) &- \frac{1}{s} X_2 (x_{\infty}; \lambda)^* P(s) \varphi' (1) = 0. 
\end{aligned}
\end{equation} 
Suppose $\varphi$ solves (\ref{left-ode+again}) for $\lambda = -\lambda_{\infty}$. Taking an 
inner product of $\varphi$ with (\ref{left-ode+again}), we get 
\begin{equation*}
- \int_0^1 ((P(\xi s)\varphi')', \varphi) d\xi + s^2 \int_0^1 (V(\xi s) \varphi, \varphi) d\xi
= s^2 \lambda \int_0^1 (Q(\xi s) \varphi, \varphi) d\xi. 
\end{equation*}
For the first integral, we have 
\begin{equation*}
- \int_0^1 ((P(\xi s)\varphi')', \varphi) d\xi
= \int_0^1 (P(\xi s)\varphi', \varphi') d\xi
- (P(s) \varphi' (1), \varphi (1)) + (P(0) \varphi' (0), \varphi (0)).
\end{equation*}

For the boundary term at $\xi = 1$, we have 
\begin{equation*}
- (P(s) \varphi' (1), \varphi (1))
= - s ((X_2 (x_{\infty}; \lambda)^*)^{-1} Y_2 (x_{\infty}; \lambda)^* \varphi (1), \varphi (1))
\ge 0,
\end{equation*}
where the inequality follows for $x_{\infty}$ sufficiently large from our 
prior discussion of 
\begin{equation*}
(X_2 (x_{\infty}; \lambda)^*)^{-1} Y_2 (x_{\infty}; \lambda)^*.
\end{equation*}  
For the boundary term at $\xi = 0$, we have
\begin{equation*}
(P(0) \varphi' (0), \varphi (0)) = s (P_R \Lambda P_R \varphi (0), \varphi (0)).
\end{equation*}
According to Lemma 1.3.8 in \cite{BK}, we can compute the 
upper bound 
\begin{equation*}
|(P_R \Lambda P_R \varphi (0), \varphi (0))| \le C_b |\varphi (0)|^2
\le C_b (\epsilon \|\varphi'\|_{L^2 ((0,1), \mathbb{C}^n)}^2 
+ \frac{2}{\epsilon} \|\varphi\|_{L^2 ((0,1), \mathbb{C}^n)}^2).
\end{equation*}
For $\lambda < 0$, this allows us to compute 
\begin{equation*}
\begin{aligned}
s^2 \lambda \theta_Q \|\varphi\|_{L^2 ((0,1), \mathbb{C}^n)}^2 
& \ge s^2 \lambda \int_0^1 (Q(\xi s) \varphi, \varphi) d \xi \\
&\ge \theta_P \|\varphi' \|_{L^2 ((0,1), \mathbb{C}^n)}^2
- s^2 C_V \|\varphi\|_{L^2 ((0,1), \mathbb{C}^n)}^2 \\
& \quad - s C_b (\epsilon \|\varphi'\|_{L^2 ((0,1), \mathbb{C}^n)}^2 
+ \frac{2}{\epsilon} \|\varphi\|_{L^2 ((0,1), \mathbb{C}^n)}^2). 
\end{aligned}
\end{equation*}
For each $s \in [s_0, x_{\infty}]$, we choose $\epsilon = \epsilon_s = \theta_P/(s C_b)$. This ensures
\begin{equation*}
\theta_P - s C_b \epsilon = 0, 
\end{equation*}
which leads immediately to 
\begin{equation*}
s^2 \lambda \theta_Q \|\varphi\|_{L^2 ((0,1), \mathbb{C}^n)}^2 
\ge
- s^2 C_V \|\varphi\|_{L^2 ((0,1), \mathbb{C}^n)}^2 \\
 - s^2 \frac{2C_b^2}{\theta_P} \|\varphi\|_{L^2 ((0,1), \mathbb{C}^n)}^2. 
\end{equation*}
We conclude a lower bound on $\lambda$, 
\begin{equation*}
\lambda \ge - \frac{C_V}{\theta_Q} - \frac{2C_b^2}{\theta_P \theta_Q}.
\end{equation*}

Combining these observations, we can conclude that for any value $\lambda_{\infty}$
chosen so that 
\begin{equation*}
-\lambda_{\infty} < - \frac{C_V}{\theta_Q} 
- \frac{2C_b^2}{\theta_P \theta_Q},
\end{equation*}
we will have no crossings along the left shelf. Similarly as in the proof
of Theorem \ref{target0}, this leaves open the possibility of a conjugate 
point at $(0, -\lambda_{\infty})$, corresponding with an intersection 
between $\ell_1 (0; -\lambda_{\infty})$ and $\ell_2 (x_{\infty}, -\lambda_{\infty})$.
Precisely as in the proof of Theorem \ref{target0}, we can increase $\lambda_{\infty}$
(if necessary) to ensure that 
$\ell_1 (0; -\lambda_{\infty}) \cap \ell_2^+ (-\lambda_{\infty}) = \{0\}$,
and then we can choose $x_{\infty}$ sufficiently large to ensure that this 
implies $\ell_1 (0; -\lambda_{\infty}) \cap \ell_2 (x_{\infty}; - \lambda_{\infty}) = \{0\}$.
For these choices of $x_{\infty}$ and $\lambda_{\infty}$, we have 
\begin{equation*}
\mas (\ell_1 (\cdot; -\lambda_{\infty}), \ell_2 (x_{\infty}; -\lambda_{\infty}); [0, x_{\infty}]) = 0.
\end{equation*}

{\it Top shelf.} In the case of Theorem \ref{target+}, $\tilde{W} (x; \lambda)$ has
been constructed so that conjugate points along the top shelf correspond 
precisely with eigenvalues of $\mathcal{L}$. In order to verify that the 
Maslov index along the top shelf corresponds with a count of eigenvalues, 
we need to check that the eigenvalues of $\tilde{W} (x; \lambda)$ rotate 
monotonically counterclockwise as $\lambda$ decreases. In this case, both 
$\mathbf{X_1}$ and $\mathbf{X}_2$ depend on $\lambda$, so according to 
Lemma 3.1 of \cite{HS2} (also Lemma 4.2 of \cite{HLS1}), rotation of the eigenvalues of 
$\tilde{W} (x; \lambda)$---for any $x \in [0, \infty)$---can be determined from the 
matrices $- \mathbf{X}_1 (x; \lambda)^* J \partial_{\lambda} \mathbf{X}_1 (x; \lambda)$
and $\mathbf{X}_2 (x_{\infty}; \lambda)^* J \partial_{\lambda} \mathbf{X}_2 (x_{\infty}; \lambda)$
in the following sense: If {\it both} of these matrices are non-positive, and 
at least one is negative definite at some point $(x_0, \lambda_0)$, then as $\lambda$ increases
through $\lambda_0$ (with $x = x_0$ fixed), all $n$ eigenvalues of 
$\tilde{W} (x; \lambda)$ will monotonically rotate in the 
clockwise direction. 

We have already established during the proof of Theorem \ref{target0} that 
the matrix 
\begin{equation*}
\mathbf{X}_2 (x_{\infty}; \lambda)^* J \partial_{\lambda} \mathbf{X}_2 (x_{\infty}; \lambda)
\end{equation*}
is negative definite, so we only need to check that 
$- \mathbf{X}_1 (x; \lambda)^* J \partial_{\lambda} \mathbf{X}_1 (x; \lambda)$
is non-positive. In fact, this latter matrix is negative definite as well, and since 
the proof is essentially identical to the proof for 
$\mathbf{X}_2 (x_{\infty}; \lambda)^* J \partial_{\lambda} \mathbf{X}_2 (x_{\infty}; \lambda)$,
we omit the details. 

We can conclude, similarly as for the bottom shelf in the proof 
of Theorem \ref{target0}, that 
\begin{equation*}
\mas (\ell_1 (x_{\infty}; \cdot), \ell_2 (x_{\infty}; \cdot); [-\lambda_{\infty}, \lambda_0]) 
= - \mor (\mathcal{L}; \lambda_0).
\end{equation*} 

{\it Bottom and right shelves.} We will need to compute Maslov 
indices along the bottom and right shelves, so it's natural to 
address the two of them together. Our approach is based 
substantially on the proofs of Claims 4.11 and 4.12 in 
\cite{HLS2}.

As a starting point, we introduce the new unitary matrix 
\begin{equation*}
\begin{aligned}
\tilde{\mathcal{W}} (x; \lambda) &:= 
- (X_1 (x; \lambda) + i Y_1 (x; \lambda)) (X_1 (x; \lambda) - i Y_1 (x; \lambda))^{-1} \\
& \times (R(\lambda) - i S (\lambda)) (R(\lambda) + i S(\lambda))^{-1},
\end{aligned}
\end{equation*}
which detects intersections between $\ell_1 (x; \lambda)$ and the 
asymptotic Lagrangian subspace 
\begin{equation*}
\ell_2^+ (\lambda) := \lim_{x \to +\infty} \ell_2 (x; \lambda).
\end{equation*}
Likewise, we specify the asymptotic matrix 
\begin{equation*}
\tilde{\mathcal{W}}^+ (\lambda) := \lim_{x \to \infty} \tilde{\mathcal{W}} (x; \lambda),
\end{equation*}
which is well-defined for each $\lambda < \kappa$, but not generally continuous
as a function of $\lambda$. (See the appendix in \cite{HLS2} for a discussion of
this discontinuity.) 
Since $R(\lambda)$ and $S(\lambda)$ can be written down explicitly, 
it is much more convenient to work with $\tilde{\mathcal{W}} (x; \lambda)$
than it is to work with $\tilde{W} (x; \lambda)$. In light of this, we
will show that our calculations can be carried out entirely in terms
of the former matrix. In particular, we have the following claim:

\begin{claim} Under the assumptions of Theorem \ref{target+}, 
we have the relation 
\begin{equation*}
\begin{aligned}
\mas &(\ell_1 (0; \cdot), \ell_2 (x_{\infty}, \cdot); [-\lambda_{\infty}, \lambda_0])
+ \mas (\ell_1 (\cdot; \lambda_0), \ell_2 (x_{\infty}; \lambda_0); [0, x_{\infty}]) \\
&= 
\mas (\ell_1 (0; \cdot), \ell_2^+ (\cdot); [-\lambda_{\infty}, \lambda_0])
+ \mas (\ell_1 (\cdot; \lambda_0), \ell_2^+ (\lambda_0); [0, \infty)).  
\end{aligned}
\end{equation*}
\end{claim}

\begin{proof}
First, it's clear that we have the relation 
\begin{equation*}
\begin{aligned}
\tilde{\mathcal{W}} (x; \lambda) &= \tilde{W} (x; \lambda)
(X_2 (x_{\infty}; \lambda) + i Y_2 (x_{\infty}; \lambda)) (X_2 (x_{\infty}; \lambda) - i Y_2 (x_{\infty}; \lambda))^{-1} \\
&\quad \times (R(\lambda) - i S (\lambda)) (R(\lambda) + i S(\lambda))^{-1}.
\end{aligned}
\end{equation*}
Recalling from Lemma \ref{ode-lemma} that 
\begin{equation*}
\mathbf{X}_2 (x_{\infty}; \lambda) = {R(\lambda) \choose S(\lambda)} 
+ \mathbf{O} (e^{-\tilde{\eta} x_{\infty}}), 
\end{equation*} 
for some $\tilde{\eta} > 0$, we see that by choosing $x_{\infty}$ 
sufficiently large, we can ensure 
that the eigenvalues of $\tilde{\mathcal{W}} (x; \lambda)$ are as close
as we like to the eigenvalues of $\tilde{W} (x; \lambda)$ for all 
$(x,\lambda) \in [0, \infty) \times [-\lambda_{\infty}, \lambda_0]$.
(Here, exponential decay in $x$ allows us to compactify $[0,\infty)$
with the usual one-point compactification.) In particular, we can 
ensure that no eigenvalue of $\tilde{W} (x; \lambda_0)$
can complete a loop of $S^1$ unless a corresponding eigenvalue of 
$\tilde{\mathcal{W}} (x; \lambda_0)$ completes a loop of $S^1$, with
the converse holding as well. 

Following our discussion of the left shelf, we have chosen $\lambda_{\infty}$
so that 
\begin{equation*}
\ell_1 (0; -\lambda_{\infty}) \cap \ell_2^+ (-\lambda_{\infty}) = \{0\},
\end{equation*}
and $x_{\infty}$ sufficiently large to ensure that this 
implies 
\begin{equation*}
\ell_1 (0; -\lambda_{\infty}) \cap \ell_2 (x_{\infty}; - \lambda_{\infty}) 
= \{0\}.
\end{equation*} 
With these choices, we see that $\tilde{W} (0; - \lambda_{\infty})$ does not
have $-1$ as an eigenvalue, and also $\tilde{\mathcal{W}} (0; -\lambda_{\infty})$
does not have $-1$ as an eigenvalue.

{\it Case 1.} First, suppose $\lambda_0$ is not an eigenvalue for 
$\mathcal{L}$. Then $\tilde{W} (x_{\infty}; \lambda_0)$ does not
have $-1$ as an eigenvalue, and also $\tilde{\mathcal{W}}^+ (\lambda_0)$
does not have $-1$ as an eigenvalue. By continuity, we can take 
$x_{\infty}$ large enough so that $\tilde{\mathcal{W}} (x_{\infty}; \lambda_0)$
does not have $-1$ as an eigenvalue, and additionally so that 
$\tilde{\mathcal{W}} (x; \lambda_0)$ does not have $-1$ as an 
eigenvalue for any $x \ge x_{\infty}$. Since the eigenvalues of 
$\tilde{W}$ and $\tilde{\mathcal{W}}$ remain uniformly close, 
the total spectral flow associated with the bottom 
and right shelves for $\tilde{W} (x; \lambda)$ must be 
the same as for $\tilde{\mathcal{W}} (x; \lambda)$. Specifically, 
we have 
\begin{equation*}
\begin{aligned}
\mas &(\ell_1 (0; \cdot), \ell_2 (x_{\infty}, \cdot); [-\lambda_{\infty}, \lambda_0])
+ \mas (\ell_1 (\cdot; \lambda_0), \ell_2 (x_{\infty}; \lambda_0); [0, x_{\infty}]) \\
&= 
\mas (\ell_1 (0; \cdot), \ell_2^+ (\cdot); [-\lambda_{\infty}, \lambda_0])
+ \mas (\ell_1 (\cdot; \lambda_0), \ell_2^+ (\lambda_0); [0, x_{\infty}]),
\end{aligned}
\end{equation*}
and the claim for Case 1 follows immediately from the specification 
that $x_{\infty}$ is taken large enough so that $\ell_1 (x; \lambda_0)$
and $\ell_2^+ (\lambda_0)$ do not intersect for $x \ge x_{\infty}$.

{\it Case 2.} Next, suppose $\lambda_0$ is an eigenvalue for $\mathcal{L}$. Then certainly
$\tilde{W} (x_{\infty}; \lambda_0)$ has $-1$ as an eigenvalue, and its 
multiplicity corresponds with the multiplicity of $\lambda_0$ as an eigenvalue
of $\mathcal{L}$. Likewise, $\tilde{\mathcal{W}}^+ (\lambda_0)$ will have 
$-1$ as an eigenvalue, and its multiplicity corresponds with the multiplicity 
of $\lambda_0$ as an eigenvalue of $\mathcal{L}$. As in the case when $\lambda_0$
is not an eigenvalue, we can choose $x_{\infty}$ large enough so that for 
$x \ge x_{\infty}$ the eigenvalues of $\tilde{\mathcal{W}} (x; \lambda)$ 
that do not approach $-1$ as $x \to +\infty$ remain bounded away from $-1$
as $x \to +\infty$. 

We now proceed precisely as in Case 1 for the eigenvalues of 
$\tilde{W} (x_{\infty}; \lambda_0)$ other than $-1$, and we note that 
an eigenvalue of $\tilde{W} (x; \lambda_0)$ will approach $-1$ as 
$x \to x_{\infty}$ if and only if an eigenvalue of 
$\tilde{\mathcal{W}} (x; \lambda)$ approaches $-1$ as $x \to +\infty$. 
Moreover, despite possible transient crossings, the net number of crossings
associated with these eigenvalues must coincide, because otherwise, an 
eigenvalue of $\tilde{W} (x; \lambda)$ would complete a full loop of $S^1$
without a corresponding eigenvalue of $\tilde{\mathcal{W}} (x; \lambda)$ 
also completing such a loop (or vice versa). 
\end{proof}

Combining now our observations for the four shelves, we find that the 
sum 
\begin{equation*}
\text{bottom shelf}
+ \text{right shelf}
+ \text{top shelf}
+ \text{left shelf} = 0,
\end{equation*}
respectively becomes
\begin{equation*}
\mas (\ell_1 (0; \cdot), \ell_2^+ (\cdot); [-\lambda_{\infty}, \lambda_0])
+ \mas (\ell_1 (\cdot; \lambda_0), \ell_2^+ (\lambda_0); [0, \infty)) 
+ \mor (\mathcal{L}; \lambda_0) + 0 
= 0,
\end{equation*}
and Theorem \ref{target+} is just a rearrangement of this equality. 
\hfill $\square$

\subsection{Changing the Target} \label{changing-target}

In this section, we verify that under certain conditions the target frame 
$\ell_2^+ (\lambda_0)$ in the calculation 
$\mas (\ell_1 (\cdot; \lambda_0), \ell_2^+ (\lambda_0); [0,\infty))$ 
can be replaced with the Dirichlet plane $\ell_D$. As noted 
earlier, one advantage of this replacement is that for a Dirichlet 
target the rotation of eigenvalues of $\tilde{W} (x; \lambda)$ as 
$x$ increases is monotonically clockwise. (This is straightforward 
to show, e.g., with the methods of \cite{HS2}.) 
The key observation we take advantage of here is that if 
$\lambda_0$ is not an eigenvalue of $\mathcal{L}$, then we 
explicitly know both $\ell_1 (0; \lambda_0)$ and 
\begin{equation*}
\ell_1^+ (\lambda_0) = \lim_{x \to +\infty} \ell_1 (x; \lambda_0)
= \tilde{\ell}_2^+ (\lambda_0),
\end{equation*}
where $\tilde{\ell}_2^+ (\lambda_0)$ denotes the Lagrangian 
subspace associated with solutions that
grow as $x$ tends to positive infinity. This allows us to 
compute both 
\begin{equation*}
\sgn Q (\ell_D, \ell_2^+ (\lambda_0); \ell_1 (0; \lambda_0))
\quad \text{and} \quad
\sgn Q (\ell_D, \ell_2^+ (\lambda_0); \ell_1^+ (\lambda_0)),
\end{equation*} 
and consequently we can compute the H\"ormander index
$s (\ell_D, \ell_2^+ (\lambda_0); \ell_1 (0; \lambda_0), \ell_1^+ (\lambda_0))$.

In order to apply our development from Section 
\ref{target-exchange}, we need the following five conditions to 
hold: (i) $\ell_D \cap \ell_1 (0;\lambda_0) = \{0\}$;
(ii) $\ell_2^+ (\lambda_0) \cap \ell_1 (0;\lambda_0) = \{0\}$;
(iii) $\ell_D \cap \ell_1^+ (\lambda_0) = \{0\}$;
(iv) $\ell_2^+ (\lambda_0) \cap \ell_1^+ (\lambda_0) = \{0\}$;
and $\ell_D \cap \ell_2^+ (\lambda_0) = \{0\}$. We will check 
below that Items (iii), (iv), and (v) hold under our general 
assumptions, and we will take Items (i) and (ii) to be additional
assumptions for this section.  

The first items to check are (iii) and (iv), which we can express as
the intersections $\ell_D \cap \tilde{\ell}_2^+ (\lambda_0)  = \{0\}$
and $\ell_2^+ (\lambda_0) \cap \tilde{\ell}_2^+ (\lambda_0) = \{0\}$. For these, we 
recall that our frame for $\tilde{\ell}_2^+ (\lambda_0)$ is 
\begin{equation*}
\tilde{\mathbf{X}}_2^+ (\lambda_0) = 
{R (\lambda_0) \choose - P R(\lambda_0) D(\lambda_0)}, 
\end{equation*}
where $R(\lambda_0)$ and $D (\lambda_0)$ are as in (\ref{R-defined}) and (\ref{D-defined}). For 
the Dirichlet plane, 
\begin{equation*}
(\tilde{\mathbf{X}}_2^+ (\lambda_0))^* J \mathbf{X}_D = (R(\lambda_0)^* \,\, - D(\lambda_0)^* P R(\lambda_0)^*) {-I \choose 0}
= R(\lambda_0)^*,
\end{equation*}
and since $R(\lambda_0)$ is invertible we have $\ker \Big((\tilde{\mathbf{X}}_2^+ (\lambda_0))^* J \mathbf{X}_D\Big) = \{0\}$.
Likewise, the frame for $\ell_2^+ (\lambda_0)$ is $\mathbf{X}_2^+ (\lambda_0) = {R (\lambda_0) \choose P R(\lambda_0) D(\lambda_0)}$,
so that 
\begin{equation*}
\begin{aligned}
(\tilde{\mathbf{X}}_2^+ (\lambda_0))^* J \mathbf{X}_2^+ (\lambda_0) &= (R(\lambda_0)^* \,\, - D(\lambda_0)^* R(\lambda_0)^* P) {-P R(\lambda_0) D(\lambda_0) \choose R(\lambda_0)} \\
&= - R(\lambda_0)^* P R(\lambda_0) D(\lambda_0) - D(\lambda_0)^* R(\lambda_0)^* P R(\lambda_0) \\
&= - 2 D(\lambda_0),
\end{aligned}
\end{equation*}
which is positive definite. The verification that $\ell_D \cap \ell_2^+ (\lambda_0) = \{0\}$ (i.e., Item (v) above) is 
essentially identical to the verification that $\ell_D \cap \tilde{\ell}_2^+ (\lambda_0) = \{0\}$, 
and we omit the details. 

By definition, the H\"ormander index for these Lagrangian subspaces
is 
\begin{equation*} 
s (\ell_D, \ell_2^+ (\lambda_0); \ell_1 (0; \lambda_0), \tilde{\ell}_2^+ (\lambda_0))
= \mas (\ell_1 (\cdot; \lambda_0), \ell_2^+ (\lambda_0); [0,\infty)) - \mas (\ell_1 (\cdot;\lambda_0), \ell_D; [0,\infty)).
\end{equation*}
According to H\"ormander's formula (\ref{hormander-formula}),
\begin{equation}
s (\ell_D, \ell_2^+ (\lambda_0); \ell_1 (0; \lambda_0), \tilde{\ell}_2^+ (\lambda_0)) =
\frac{1}{2} \Big(\sgn Q (\ell_D, \ell_2^+ (\lambda_0); \ell_1 (0; \lambda_0)) 
- \sgn Q (\ell_D, \ell_2^+ (\lambda_0); \tilde{\ell}_2^+ (\lambda_0))\Big).
\end{equation}
We can now use Lemma \ref{hormander-lemma} to compute the two quantities
$\sgn Q (\ell_D, \ell_2^+ (\lambda_0); \ell_1 (0; \lambda_0))$ and 
$\sgn Q (\ell_D, \ell_2^+ (\lambda_0); \tilde{\ell}_2^+ (\lambda_0))$. First, recalling
that the frame for $\ell_1 (0; \lambda_0)$ is 
$\mathbf{X}_1 (0;\lambda_0) = {-\alpha_2^* \choose \alpha_1^*}$,
and noting that the condition $\ell_1 (0; \lambda_0) \cap \ell_D = \{0\}$
implies that $\alpha_2$ is invertible,
we have (from Lemma \ref{hormander-lemma})
\begin{equation*}
\begin{aligned}
\sgn Q (\ell_D, \ell_2^+ (\lambda_0); \ell_1 (0; \lambda_0)) 
&= \sgn  (- \alpha_1^* (\alpha_2^*)^{-1} - Y_2^+ (X_2^+)^{-1}) \\
&= \sgn (- \alpha_1^* (\alpha_2^*)^{-1} - P_+ R(\lambda_0) D(\lambda_0) R(\lambda_0)^{-1}) \\
&= \sgn (- \alpha_1^* (\alpha_2^*)^{-1} - P_+ R(\lambda_0) D(\lambda_0) R(\lambda_0)^* P_+).
\end{aligned}
\end{equation*}
Likewise, 
\begin{equation*}
\begin{aligned}
\sgn Q (\ell_D, \ell_2^+ (\lambda_0); \tilde{\ell}_2^+ (\lambda_0)) 
&= \sgn  (- Y_2^+ (\lambda_0) (X_2^+ (\lambda_0))^{-1} - Y_2^+ (\lambda_0) (X_2^+ (\lambda_0))^{-1}) \\
&= \sgn (- 2 P_+ R(\lambda_0) D(\lambda_0) R(\lambda_0)^* P_+) \\
&= n,
\end{aligned}
\end{equation*}
because $P_+ R(\lambda_0) D(\lambda_0) R(\lambda_0)^* P_+$ is negative definite. 

Combining these observations, we see that 
\begin{equation*}
\begin{aligned}
\mas (\ell_1 (\cdot; \lambda_0), \ell_2^+ (\lambda_0); [0,\infty)) &= 
\mas (\ell_1 (\cdot; \lambda_0), \ell_D; [0,\infty)) \\
& + 
\frac{1}{2} \Big(-n + \sgn(- \alpha_1^* (\alpha_2^*)^{-1} - P_+ R(\lambda_0) D(\lambda_0) R(\lambda_0)^* P)\Big).
\end{aligned}
\end{equation*}

In this way, we obtain the following corollary to Theorem \ref{target+}. 

\begin{corollary} \label{corollary+}
Let the assumptions of Theorem \ref{target+} hold, and suppose additionally 
that $\lambda_0 \notin \sigma (\mathcal{L})$, $\ell_1 (0; \lambda_0) \cap \ell_D = \{0\}$,
and $\ell_1 (0; \lambda_0) \cap \ell_2^+ = \{0\}$. Then
\begin{equation*}
\begin{aligned}
\mor (\mathcal{L}; \lambda_0)
&= - \mas (\ell_1 (\cdot; \lambda_0), \ell_D; [0, \infty)) \\
& - \frac{1}{2} \Big(-n + \sgn(- \alpha_1^* (\alpha_2^*)^{-1} - P_+ R(\lambda_0) D(\lambda_0) R(\lambda_0)^* P)\Big) \\
& - \mas (\ell_1 (0; \cdot), \ell_2^+ (\cdot); [-\lambda_{\infty}, \lambda_0]).
\end{aligned}
\end{equation*} 
\end{corollary}

\section{Application to Quantum Graphs} \label{application-graphs}

In this section, we apply our framework to an operator on the half-line
that arises through consideration of nonlinear Schr\"odinger 
equations on quantum graphs with $n$ infinite edges extending from 
a single vertex (i.e., on star graphs). Our direct motivation for considering this example 
is the recent analysis of Kairzhan and Pelinovsky (see \cite{KP2018}),
and we also note that Kostrykin and Schrader have shown how the 
symplectic framework fits well with such problems (see \cite{KS99}) and 
that Latushkin and Sukhtaiev have recently developed this framework 
in the case of quantum graphs with edges of finite length (see \cite{LS2018}).
Finally, we mention that our general approach to quantum graphs
is adapted from the reference \cite{BK}. 

\subsection{The Schr\"odinger Operator on Star Graphs}

We consider a star graph with $n$ edges, which can be visualized 
as a single point with $n$ distinct half-lines emerging from it.
We will associate with each edge of our graph the interval 
$[0, \infty)$, and our basic Hilbert space associated with 
the full graph will be 
\begin{equation*}
\mathcal{H} = \bigoplus_{j=1}^n L^2 ((0, \infty), \mathbb{C}).
\end{equation*}
We will view elements $\phi \in \mathcal{H}$ as vector 
functions $\phi = (\phi_1, \phi_2, \cdots, \phi_n)^t$, 
and we specify the linear operator $\mathcal{L}: \mathcal{H} \to \mathcal{H}$
by 
\begin{equation*}
(\mathcal{L} \phi)_j = - \phi_j'' + v(x) \phi_j,
\end{equation*} 
where $v \in C([0,\infty), \mathbb{C})$ is a scalar potential for 
which we will assume the limit 
\begin{equation*}
\lim_{x \to \infty} v(x) = v_+
\end{equation*} 
exists and satisfies the asymptotic relation
\begin{equation*}
\int_0^{\infty} x (v(x) - v_+) dx < \infty.
\end{equation*}
(This is slightly weaker than our Assumption {\bf (A2)}, but 
sufficient in the current setting (see \cite{HLS2}).)
We specify boundary conditions at the vertex as 
\begin{equation} \label{graph-bc}
\alpha_1 \phi (0) + \alpha_2 \phi'(0) = 0,
\end{equation}
with $\alpha_1$ and $\alpha_2$ satisfying the assumptions 
described in {\bf (A3)}. Under these assumptions, we take as 
our domain for $\mathcal{L}$, 
\begin{equation*}
\mathcal{D} (\mathcal{L}) = 
\{\phi \in \mathcal{H}: \phi, \phi' \in AC_{\loc}([0,\infty), \mathbb{C}^n), \,
\mathcal{L} \phi \in \mathcal{H} \}. 
\end{equation*}
 
With this notation in place, we can consider the eigenvalue 
problem $\mathcal{L} \phi = \lambda \phi$ with boundary 
conditions (\ref{graph-bc}). In order to place this system 
in the framework of our analysis, we set 
$y(x; \lambda) = {y_1 (x; \lambda) \choose y_2 (x; \lambda)}$,
with $y_1 (x; \lambda) = \phi (x; \lambda)$ and 
$y_2 (x; \lambda) = \phi' (x; \lambda)$. In this way, we arrive
at our standard Hamiltonian system 
\begin{equation} \label{graph-hammy}
\begin{aligned}
Jy' &= \mathbb{B} (x; \lambda) y \\
\alpha y(0) &= 0,
\end{aligned}
\end{equation}
where $\mathbb{B} (x; \lambda)$ denotes the diagonal matrix
\begin{equation*}
\mathbb{B} (x; \lambda) = 
\begin{pmatrix}
(\lambda - v(x))I & 0 \\
0 & I
\end{pmatrix}.
\end{equation*}

Under our assumptions on the scalar potential $v$,
it's well known that for each $\lambda < v_+$ the 
scalar equation 
\begin{equation} \label{scalar-equation}
- z'' + v(x) z = \lambda z
\end{equation}
has one non-trivial solution that decays as 
$x \to +\infty$ and one non-trivial solution 
that grows as $x \to +\infty$. (See, e.g., 
\cite{HLS2}.) If we denote by $\zeta (x; \lambda)$
the solution that decays as $x \to + \infty$,
then we can express our frame $\mathbf{X}_2 (x;\lambda)$
of solutions of (\ref{graph-hammy}) decaying as $x \to + \infty$ as 
\begin{equation*}
\mathbf{X}_2 (x; \lambda) = 
\begin{pmatrix}
\zeta (x; \lambda) I \\
\zeta' (x; \lambda) I
\end{pmatrix}.
\end{equation*}
We see that in this case, and in the context of Theorem \ref{target0}, 
\begin{equation*}
\tilde{W} (x; \lambda) = 
- (-\alpha_2^* + i \alpha_1^*) (- \alpha_2^* - i \alpha_1^*)^{-1}
\frac{\zeta (x; \lambda) - i \zeta' (x; \lambda)}{\zeta (x; \lambda) + i \zeta' (x; \lambda)}. 
\end{equation*}
(I.e., this is (\ref{tildeW0}) for the current case.)
In particular, if we denote the eigenvalues of 
$(-\alpha_2^* + i \alpha_1^*) (- \alpha_2^* - i \alpha_1^*)^{-1}$ 
by $\{a_j\}_{j=1}^n$, then the eigenvalues of $\tilde{W} (x; \lambda)$
will be 
\begin{equation*}
\Big{\{}
- \frac{\zeta (x; \lambda) - i \zeta' (x; \lambda)}{\zeta (x; \lambda) + i \zeta' (x; \lambda)} a_j
\Big{\}}_{j=1}^n.
\end{equation*}

\begin{remark} \label{neumann-bc} We distinguish the {\it Neumann} or 
{\it Neumann-Kirchhoff} boundary conditions as those specified by 
the relations 
\begin{equation*}
\begin{aligned}
\phi_1 (0) &= \phi_2 (0) = \dots = \phi_n (0) = 0 \\
\sum_{j=1}^n \phi_j (0) &= 0. 
\end{aligned}
\end{equation*}
(See p. 14 of \cite{BK} for a discussion of terminology.) These 
correspond with 
\begin{equation} \label{alpha1}
\alpha_1 = 
\begin{pmatrix}
1 & -1 & 0 & \cdots & 0 & 0 \\
0 & 1 & -1 & \cdots & 0 & 0\\
\vdots & \vdots & \vdots & \vdots & \vdots & \vdots \\
0 & 0 & 0 & \cdots & 1 & -1 \\
0 & 0 & 0 & \cdots & 0 & 0  
\end{pmatrix}
\end{equation}
and 
\begin{equation} \label{alpha2}
\alpha_2 = 
\begin{pmatrix}
0 & 0 & \cdots & 0 \\
\vdots & \vdots & \vdots & \vdots \\
0 & 0 & \cdots & 0 \\
1 & 1 & \cdots & 1 
\end{pmatrix}.
\end{equation}
In this case, the eigenvalues of 
$(-\alpha_2^* + i \alpha_1^*) (- \alpha_2^* - i \alpha_1^*)^{-1}$ 
are $-1$ and $+1$, with $+1$ simple and $-1$ occurring with 
multiplicity $n-1$. This fact is straightforward to verify 
directly, and is also an immediate consequence of 
Corollary 2.3 from \cite{KS99}. 
\end{remark}

\subsection{NLS on Star Graphs}

We now consider the nonlinear Schr\"odinger equation 
\begin{equation} \label{nls}
i u_t = - \Delta u - (p+1) |u|^{2p} u,
\end{equation}
where $p > 0$ and $u \in \mathcal{H}$ with $u_j$ taking the 
values of $u$ on edge $j$ of the graph. We interpret the
notation $\Delta u$ and $|u|^{2p} u$ in this setting as
\begin{equation*}
\begin{aligned}
\Delta u &= (u_1'', u_2'', \dots, u_n'')^t \\
|u|^{2p} u &= (|u_1|^{2p} u_1, |u_2|^{2p} u_2, \dots, |u_n|^{2p} u_n)^t. 
\end{aligned}
\end{equation*} 
Such equations are known to admit standing wave solutions 
\begin{equation*}
u (x, t) = e^{i \omega t} \tilde{u}_{\omega} (x),
\end{equation*}
for any $\omega > 0$. Upon direct substitution into 
(\ref{nls}), we see that 
\begin{equation*}
-\Delta \tilde{u}_{\omega} - (p+1) |\tilde{u}_{\omega}|^{2p} \tilde{u}_{\omega}
= - \omega \tilde{u}_{\omega}.
\end{equation*}
In \cite{KP2018}, the authors observe that by setting 
\begin{equation*}
z = \omega^{1/2} x; \quad
\tilde{u}_{\omega} (x) = \omega^{\frac{1}{2p}} \tilde{u} (z),
\end{equation*}
we arrive at 
\begin{equation} \label{omega1}
-\Delta \tilde{u} - (p+1) |\tilde{u}|^{2p} \tilde{u}
= - \tilde{u}.
\end{equation}
This scaling justifies restricting our attention to the case 
$\omega = 1$. It's straightforward to verify that for any $p > 0$ 
(\ref{omega1}) admits the explicit solution
\begin{equation*}
\tilde{u} (x) = 
s(x) \begin{pmatrix}
1 \\
1 \\
\vdots \\
1
\end{pmatrix}; 
\quad s(x) = \sech^{1/p} (px).  
\end{equation*} 

We linearize (\ref{nls}) about $e^{it} \tilde{u} (x)$, 
writing 
\begin{equation*}
u(x, t) = e^{it} \tilde{u} (x) + e^{it} (v(x, t) + i w (x, t)),
\end{equation*} 
where $v(x,t)$ and $w (x, t)$ are both real-valued functions. 
Dropping off higher order terms, we obtain the linear system
\begin{equation*}
\begin{aligned}
v_t &= L_- w \\
w_t &= - L_+ v,
\end{aligned}
\end{equation*}
where 
\begin{equation*}
\begin{aligned}
L_- &= - \Delta + 1 - (p+1) \tilde{u} (x)^{2p}  \\
L_+ &= -\Delta + 1 - (p+1) (2p+1) \tilde{u} (x)^{2p}.
\end{aligned}
\end{equation*}

Our framework can now be used in order to determine 
the Morse indices of $L_{\pm}$ with Neumann--Kirchhoff 
boundary conditions. We focus on the slightly more 
interesting case, $L_+$. (The Morse index of $L_-$
with Neumann--Kirchhoff boundary conditions is $0$.)
The eigenvalue problem for $L_+$ can be expressed 
as 
\begin{equation} \label{nls-evp}
\begin{aligned}
- \phi'' + (1 - (p+1) (2p+1) s(x)^{2p}) \phi &= \lambda \phi; \quad
x \in (0, \infty) \\
\alpha_1 \phi (0) + \alpha_2 \phi'(0) &=0, 
\end{aligned}
\end{equation}
with $\alpha_1$ and $\alpha_2$ as expressed in 
(\ref{alpha1}) and (\ref{alpha2}). 

For this calculation, we will use Theorem \ref{target0} with 
$\lambda_0 = 0$. We observe that by construction, 
\begin{equation*}
\phi (x) = 
s'(x) \begin{pmatrix}
1 \\
1 \\
\vdots \\
1
\end{pmatrix}; 
\quad s(x) = \sech^{1/p} (px),  
\end{equation*}  
solves (\ref{nls-evp}) for $\lambda = 0$ (just differentiate
(\ref{omega1}) to see this; here, $\phi$ is not expected to 
satisfy the boundary condition at $x = 0$). This allows us to 
express our frame for solutions
of (\ref{nls-evp}) that decay as $x \to +\infty$ as 
\begin{equation*}
\mathbf{X}_2 (x; \lambda)
= {s' (x) I \choose s''(x) I}. 
\end{equation*} 
We set $\mathbf{X}_1 (0; \lambda) = {- \alpha_2^* \choose \alpha_1^*}$,
so that 
\begin{equation*}
\tilde{W} (x; 0) = 
- (-\alpha_2^* + i \alpha_1^*) (- \alpha_2^* - i \alpha_1^*)^{-1}
\frac{s'(x) - i s''(x)}{s'(x) + i s'' (x)}.
\end{equation*}
According to Remark \ref{neumann-bc}, the eigenvalues of 
$\tilde{W} (x; 0)$ are 
\begin{equation*}
q(x) := (s'(x) - i s''(x))(s'(x) + i s'' (x))^{-1},
\end{equation*}
with multiplicity $n - 1$ and the negative of this with
multiplicity 1. (Here, the notation $q(x)$ has been introduced simply 
for expositional convenience). 

In \cite{HS2}, the authors have developed a straightforward 
approach toward determining the direction of rotation for the 
eigenvalues of $\tilde{W} (x; \lambda)$ as $x$ varies, but 
in the current setting this rotation can be determined directly
from the form of $s (x)$. We observe that
\begin{equation*}
\begin{aligned}
s'(x) &= - s(x) \tanh (px) \\
s''(x) &= s(x) \tanh^2 (px) - s(x) p \sech^2 (px). 
\end{aligned}
\end{equation*}
We can write 
\begin{equation*}
\frac{s'(x) - i s''(x)}{s'(x) + i s'' (x)} 
= \frac{s'(x)^2 - s''(x)^2 - 2i s'(x) s''(x)}{s'(x)^2 + s''(x)^2}, 
\end{equation*}
for which we focus on the real and imaginary parts of the numerator 
\begin{equation*}
\begin{aligned}
s'(x)^2 - s''(x)^2 &= 
s(x)^2 \Big(\tanh^2 (px) - (\tanh^2 (px) - p \sech^2 (px))^2 \Big) \\
-2 s'(x) s''(x) &= 2 s(x)^2 \tanh (px) \Big(\tanh^2 (px) - p \sech^2 (px)\Big).
\end{aligned}
\end{equation*}
We note that for any $x > 0$,
\begin{equation} \label{reim}
\begin{aligned}
\sgn \text{Re }q(x) &= \sgn \Big(\tanh^2 (px) - (\tanh^2 (px) - p \sech^2 (px))^2 \Big) \\
\sgn \text{Im }q(x) &= \sgn \Big(\tanh^2 (px) - p \sech^2 (px)\Big).
\end{aligned}
\end{equation}

We now consider the motion of $q(x)$
as $x$ runs from $0$ to $+\infty$. 
First, $s' (0) = 0$ and $s'' (0) = -p$, so 
\begin{equation*}
q(0) = -1.
\end{equation*}
This means that $-1$ is an eigenvalue of $\tilde{W} (0; 0)$
with multiplicity $n-1$, and $+1$ is an eigenvalue of 
$\tilde{W} (0; 0)$ with multiplicity $1$. 
As $x$ increases from 0, we see from (\ref{reim}) that the 
imaginary part of $q(x)$ is negative, so rotation is in the 
counterclockwise direction. Moreover, since $\tanh^2 (px)$
and $\sech^2 (px)$ are both monotonic in $x$ (for $x \ge 0$),
we see that the imaginary part of $q(x)$ remains negative
until $x$ arrives at the unique value $\bar{x}$ for which
\begin{equation*}
\tanh^2 (p \bar{x}) - p \sech^2 (p\bar{x}) = 0.
\end{equation*}
We see from (\ref{reim}) that $\sgn \text{Re }q(\bar{x}) > 0$,
so $q (\bar{x}) = +1$. For $x > \bar{x}$, the imaginary part
of $q(x)$ is positive, and by noting the asymptotic relations 
$s'(x) \sim - 2^{1/p} e^{-x}$, $s'' (x) \sim 2^{1/p} e^{-x}$,
we see that as $x \to + \infty$, $q(x)$ approaches $i$. In 
summary, we see that as $x$ runs from $0$ to $+\infty$, $q(x)$
rotates from $-1$ to $i$, leaving $-1$ in the counterclockwise 
direction and never crossing $-1$. Indeed, with a bit more work,
we can verify that the rotation is entirely counterclockwise, 
but we don't require that much information to draw our 
conclusions. 

Returning to the matrix $\tilde{W} (x; 0)$, we can conclude that 
$n-1$ eigenvalues trace out precisely the path described in 
the previous paragraph, and the final eigenvalue begins at 
$+1$ when $x=0$ and rotates in the counterclockwise direction, 
approaching $-i$ as $x \to +\infty$. We conclude that
\begin{equation*}
\mas (\ell_1 (0; 0), \ell_2 (\cdot; 0); [0, \infty)) = +1.
\end{equation*} 

Finally, in order to use Theorem \ref{target0}, we need to compute 
$\mas (\ell_1 (0; \cdot), \ell_2^+ (\cdot); [-\lambda_{\infty}, 0])$.
For this, we observe that if we set $y = {y_1 \choose y_2}$, 
with $y_1 = \phi$ and $y_2 = \phi'$, then (\ref{nls-evp})
can be expressed as $y' = \mathbb{A} (x; \lambda) y$,
with 
\begin{equation*}
\mathbb{A} (x; \lambda) 
= \begin{pmatrix}
0 & I_n \\
((1 - (p+1) (2p+1) s(x)^{2p}) - \lambda)I_n & 0 
\end{pmatrix}.
\end{equation*}
Since $s(x) \to 0$ as $x \to \infty$, we see that 
\begin{equation*}
\mathbb{A}_+ (\lambda) :=
\lim_{x \to \infty} \mathbb{A} (x; \lambda) 
= \begin{pmatrix}
0 & I_n \\
(1 - \lambda)I_n & 0 
\end{pmatrix}.
\end{equation*}
We can readily check that as a choice for the corresponding 
asymptotic frame $\mathbf{X}_2^+ (\lambda) = {R (\lambda) \choose S (\lambda)}$, 
we can take $\mathbf{X}_2^+ (\lambda) = {I_n \choose - \sqrt{1 - \lambda} I_n}$. 
Thus for the top shelf, we have 
\begin{equation*}
\tilde{W}^+ (\lambda) = - (-\alpha_2^* + i \alpha_1^*) (- \alpha_2^* - i \alpha_1^*)^{-1}
\frac{1+i\sqrt{1-\lambda}}{1-i\sqrt{1-\lambda}}.
\end{equation*}
We conclude from Remark \ref{neumann-bc} that the eigenvalues 
of $\tilde{W}^+ (\lambda)$ are 
$(1+i\sqrt{1-\lambda})(1-i\sqrt{1-\lambda})^{-1}$ with 
multiplicity $(n-1)$ and the negative of this with multiplicity
$1$. For $\lambda < 1$ the value of 
$(1+i\sqrt{1-\lambda})(1-i\sqrt{1-\lambda})^{-1}$
cannot be $\pm 1$, so there are no conjugate points along the 
top shelf. We conclude that in this case
\begin{equation*}
\mas (\ell_1 (0; \cdot), \ell_2^+ (\cdot); [-\lambda_{\infty}, 0]) = 0.
\end{equation*}

Applying Theorem \ref{target0}, we find that 
\begin{equation*}
\mor (L_+) = \mas (\ell_1 (0; 0), \ell_2 (\cdot; 0); [0, \infty)) = +1.
\end{equation*}

\begin{remark} For a more complete discussion of the instability 
of the half-soliton $e^{i \omega t} \tilde{u}_{\omega} (x)$ as a 
solution to (\ref{nls}), including a calculation of $\mor (L_+)$
by other means, we refer the reader to \cite{KP2018}.
\end{remark}

\bigskip
{\it Acknowledgments.} This work was initiated while P.H. was visiting
Miami University in March, 2018. The authors are grateful to the 
Department of Mathematics at Miami University for supporting this 
trip.

\end{document}